\documentclass[11pt,reqno]{amsart}

\usepackage{amsmath, amsfonts, amsthm, amssymb}

\newtheorem{Theorem}{Theorem}[section]
\newtheorem{Lemma}{Lemma}[section]
\newtheorem{Proposition}{Proposition}[section]

\theoremstyle{definition}

\theoremstyle{remark}
\newtheorem{Remark}{Remark}[section]

\numberwithin{equation}{section}

\renewcommand{\u}{{\bf u}}

\newcommand{\R}{{\mathbb R}}
\newcommand{\Dv}{{\rm div}}

\newcommand{\m}{{\bf m}}

\def\f{\frac}
\renewcommand{\O}{\Omega}

\def\D{\Delta }

\def\hf1{^\f{1}{1-\xi^2}}

\def\be{\begin{equation}}
\def\en{\end{equation}}
\def\bs{\begin{split}}
\def\es{\end{split}}

\renewcommand{\d}{{\bf d}}

\renewcommand{\a}{\alpha}

\newcommand{\ess}{{\rm ess\sup}}

\author{Dehua Wang}
\address{Department of Mathematics, University of Pittsburgh,
                           Pittsburgh, PA 15260.}
\email{dwang@math.pitt.edu}
\author{Cheng Yu}
\address{Department of Mathematics, University of Pittsburgh,
                           Pittsburgh, PA 15260.}
\email{chy39@pitt.edu}

\title[Global Solution to the  Flow of Liquid Crystals]
{Global weak solution and large-time behavior for the compressible flow of liquid crystals}


\date{April 12, 2011} %\today}

\begin{document}

\begin{abstract}
The three-dimensional equations for the compressible flow of liquid crystals are considered. An initial-boundary value problem is studied in a bounded domain with large data. The existence and large-time behavior of a global weak solution are established through a three-level approximation, energy estimates, and weak convergence for the adiabatic  exponent $\gamma>\frac32$.
\end{abstract}

\maketitle

\section{Introduction}
In this paper, we consider the following hydrodynamic system of partial differential equations for the
three-dimensional flow of nematic liquid crystals (\cite{DE,HK,Lin2}):
\begin{subequations}\label{I1}
\begin{align}
&\rho_{t}+\Dv(\rho\u)=0,\label{I1a}\\
&(\rho \u)_{t}+\Dv(\rho \u \otimes \u)+\nabla P(\rho)= \mu \Delta\u-\lambda \Dv\left(\nabla \d \odot \nabla \d-(\frac{1}{2} |\nabla\d|^{2}+F(\d)) I_{3}\right),\label{I1b}\\
&\d_{t}+\u\cdot\nabla\d=\theta(\Delta\d-f(\d)),\label{I1c}
\end{align}
\end{subequations}
where  $\rho\ge 0$ denotes the density, $\u\in\R^3$  the velocity, $\d\in\R^3$ the direction
field for the averaged macroscopic molecular orientations, and
$P=a\rho^{\gamma}$ is the pressure with constants $a>0$ and $\gamma\geq 1$.  The positive constants $\mu, \lambda, \theta$ denote  the viscosity, the competition between kinetic energy and potential energy, and the microscopic elastic relation time for the molecular orientation field, respectively.  The symbol $\otimes$ denotes the Kronecker tensor product,
$I_{3}$ is the $3\times3$ identity matrix, and $\nabla\d\odot\nabla\d$ denotes the $3\times 3$ matrix whose $ij$-th entry is $<\partial_{x_i} \d, \partial_{x_j}\d>$.  Indeed, $$\nabla \d \odot\nabla \d= (\nabla\d)^{\top}\nabla\d,$$
where $(\nabla\d)^{\top}$ denotes the transpose of the $3\times3$ matrix $\nabla\d$.
The vector-valued smooth function $f(\d)$ denotes the penalty function and has the following form:
 $$f(\d)=\nabla_{\d}F(\d),$$
where the scalar function $F(\d)$ is the bulk part of the elastic energy.
A typical example is to choose $F(\d)$ as the the Ginzburg-Landau penalization thus yielding the penalty function $f(\d)$ as:
$$F(\d)=\frac{1}{4\sigma_0^{2}}(|\d|^{2}-1)^{2}, \quad  f(\d)=\frac{1}{2\sigma_0^{2}}(|\d|^{2}-1)\d,$$
where $\sigma_0>0$ is a constant.  We refer the readers to \cite{Chan,DE,E,HK,Leslie1,Lin2} for more physical background and discussion of liquid crystals and mathematical models.

There have been many mathematical studies on the incompressible flows of liquid crystals.  In Lin-Liu
\cite{Lin2, Lin1, LL, LL2}, the global existence of weak solutions with large
initial data was proved under the condition that the orientational
configuration $\d(x,t)$ belongs to $H^2$, and the global existence
of classical solutions was also obtained if the coefficient $\mu$
is large enough in the three-dimensional spaces. The similar results
were obtained also in \cite{SL} for a different but similar model.
The global strong solution was established in Hu-Wang \cite{HW2}.
When the weak solutions are discussed, the partial regularity of the weak
solution similar to the classical theorem of Caffarelli-Kohn-Nirenberg \cite{CKN}  was obtained in \cite{LL2} (and also \cite{HKL}).
The existence of weak solutions to the density-dependent incompressible flow of liquid crystals was proved in \cite{JT}.
The compressible flow \eqref{I1} of liquid crystals is much more complicated and difficult to study mathematically due to the compressibility.
In the one-dimensional case the global existence of  smooth and weak solutions to  the compressible flow of liquid crystals was obtained in \cite {DLWW}.  Our aim of this paper is to establish the global existence of weak solutions $ (\rho,\u,\d)$ to the three-dimensional compressible flow \eqref{I1}  of liquid crystals in a bounded smooth domain $\O\subseteq R^{3}$, with the following initial-boundary conditions:
\begin{equation}\label{I2}\begin{split}
&(\rho, \rho\u, \d)|_{t=0}=(\rho_{0}(x), \m_{0}(x), \d_{0}(x)), \quad x\in\O,
\end{split}\end{equation}
and
\begin{equation} \label{I3}
\u|_{\partial\O}=0, \quad \d|_{\partial\O} =\d_{0}(x), \quad x\in\partial\O,
\end{equation}
 where
 \begin{equation*}
 \begin{split}
 & \rho_{0} \in L^{\gamma}(\O),\quad \rho_{0}\geq 0; \quad  \d_{0}\in L^{\infty}(\O)\cap H^{1}(\O);\\
 &\m_{0}\in L^{1}(\O),\quad  m_{0}=0 \text { if } \rho_{0}=0;  \quad \frac{|\m_{0}|^{2}}{\rho_{0}}\in L^{1}(\O).
 \end{split}
 \end{equation*}

When the direction field $\d$ does not appear, \eqref{I1} reduces to the compressible Navier-Stokes equations. For the compressible Navier-Stokes equations, Lions in \cite{L} introduced the concept of  renormalized solutions to overcome the difficulties of large oscillations  and proved the global existence of finite energy weak solutions for $\gamma>9/5$, and then Feireisl, et al,  in \cite{F, FNP, FP} extended the existence results to $\gamma>3/2$.
 Hu-Wang in \cite{HW} adopted  Feireisl's techniques to obtain global existence and large-time behavior of weak solutions with large initial data for the magnetohydrodynamics.
In this paper we shall study the initial-boundary value problem \eqref{I1}-\eqref{I3} for  liquid crystals  and establish the global existence and large-time behavior of weak solutions for large initial data in certain functional spaces with $\gamma>3/2$.
To achieve our goal, we will use a three-level approximation scheme similar to that in \cite{F, FNP}, which consists of
Faedo-Galerkin approximation, artificial viscosity, and artificial pressure. Then, motivated by the work of \cite{FNP}, we will show that the uniform bound of  the density $\rho^{\gamma+\alpha} $ in $L^{1}$ for some $\alpha>0$ ensures the vanishing of artificial pressure and the strong compactness of the density. To overcome the difficulty of possible large oscillation of the density, we adopt the idea of Lions and Feireisl in \cite{F, FNP, LL}  based on the weak continuity of the effective viscous flux
 for the Navier-Stokes equations. For our equations \eqref{I1} of liquid crystals, the effective viscous flux is $P-\mu\Dv \u$. For this purpose, we also need to develop some estimates to deal with the direction field and its coupling and interaction with the fluid variables. To deal with the equation \eqref{I1c} for a given $\u$, we will follow the same idea of Hu-Wang \cite{HW} to establish the solvability of the direction field. It is crucial to obtain sufficiently strong estimates on the direction field $\d$ to recover the original system \eqref{I1}. We will derive an energy inequality from \eqref{I1} directly, but it can not provide us with sufficient regularity for the direction field $\d$. Thanks to the Gagliardo-Nirenberg inequality and the maximum principle, we deduce that $\nabla\d \in L^{4}((0,T)\times\O)$ for any $T>0$ which can be used to control the strongly nonlinear terms
containing $\nabla\d$ in \eqref{I1}.  Then we finally establish the existence of global weak solution to \eqref{I1}-\eqref{I3}.
We just noticed that a similar existence result was obtained independently in \cite{LLQ}.
Motivated by \cite{FP} and \cite{HW}, we will also establish the large-time behavior of the global weak solutions.

We organize the rest of the paper as follows. In Section 2, we deduce a prior estimates from \eqref{I1}, give the definition of  finite energy weak solutions, and also state our main results. In Section 3, we discuss the solvability of the direction vector $\d$ in terms of $\u$. In Section 4, we establish the global existence of solutions to the Faedo-Galerkin approximation to \eqref{I1}. In Section 5 and Section 6, we use the uniform estimates to recover the original system by vanishing the artificial viscosity and artificial pressure respectively. In Section 7, we prove the large-time behavior of the global weak solutions.

\bigskip

%%%%%%%%%%%%%%%%%%%%%%%%%%
\section{Energy Estimates and Main Results}

In this section, we  derive some basic energy estimates for the initial-boundary problem \eqref{I1}-\eqref{I3}, introduce the notion of finite energy weak solutions in the spirit of Feireisl \cite{F, FNP},  and state the main results.

Without loss of generality, we take $\theta=a=1$.
First we formally derive the energy equality and some a priori estimates, which will play a very important role in our paper.
Multiplying \eqref{I1b} by $\u$,  integrating over $\O$, and using the boundary condition  \eqref{I3},
we obtain
\begin{equation*}
\begin{split}
&\partial_{t}\int_{\O}{\left(\frac{1}{2}\rho|\u|^{2}+\frac{\rho^{\gamma}}{\gamma-1}\right)dx}+\mu\int_{\O}|\nabla \u|^{2}dx\\
&=-\lambda\int_{\O}\Dv\left(\nabla\d\odot\nabla\d-(\frac{1}{2}|\nabla\d|^{2}+F(\d))I_{3}\right)\u dx.
\end{split}
\end{equation*}
Using the equality $$\Dv(\nabla\d \odot \nabla\d)=\nabla(\frac{1}{2}|\nabla\d|^{2})+(\nabla\d)^{\top}\cdot\Delta\d,$$
we have \begin{equation*}\begin{split}
&\int_{\O}\Dv\left(\nabla\d\odot\nabla\d-(\frac{1}{2}|\nabla\d|^{2}+F(\d))I_{3}\right)\u dx
\\ & =\int_{\O}(\nabla\d)^{\top}\cdot\Delta\d\cdot \u dx-\int_{\O}\nabla_{\d}F(\d)\u dx.
\end{split}\end{equation*}
Hence,  we obtain
\begin{equation}\label{EI1}
\begin{split}
&\partial_{t}\int_{\O}{\left(\frac{1}{2}\rho|\u|^{2}+\frac{\rho^{\gamma}}{\gamma-1}\right)dx}+\int_{\O}|\nabla \u|^{2}dx\\
&=-\lambda\int_{\O}(\nabla\d)^{\top}\cdot\Delta\d\cdot \u dx+\lambda\int_{\O}\nabla_{\d}F(\d)\u dx.
\end{split}
\end{equation}
Multiplying by $\lambda(\Delta \d-f(\d))$ the both sides of \eqref{I1c} and integrating over $\O$, we get
\begin{equation*}
\begin{split}
&-\partial_{t}\int_{\O}\lambda\frac{|\nabla\d|^{2}}{2} dx-\partial_{t}\int_{\O}\lambda F(\d)dx-\int_{\O}\lambda\nabla_{\d}F(\d)\u dx
+\lambda\int_{\O}(\nabla\d)^{\top}\cdot\Delta\d\cdot\u dx \\
&=\lambda\int_{\O}|\Delta\d-f(\d)|^{2}dx.
\end{split}
\end{equation*}
Then,  from \eqref{EI1},  we have the following energy equality to the system \eqref{I1},
\begin{equation}\label{EI2}
\begin{split}
&\partial_{t}\int_{\O}\left(\frac{1}{2}\rho|\u|^{2}+\frac{\rho^{\gamma}}{\gamma-1}
   +\frac{\lambda}{2}|\nabla\d|^{2}+\lambda F(\d)\right)dx\\
&\qquad   +\int_{\O}\left(\mu|\nabla\u|^{2}dx+\lambda|\Delta\d-f(\d)|^{2}\right)dx\\
&=0.\end{split}\end{equation}

Set
$$E(t)=\int_{\O}\left(\frac{1}{2}\rho|\u|^{2}+\frac{\rho^{\gamma}}{\gamma-1}+\frac{\lambda}{2}|\nabla \d|^{2}
+\lambda F(\d) \right)(t,x)dx,$$
and assume that $E(0)<\infty$.
From \eqref{EI2}, we have the following a priori estimates:
$$\rho |\u|^{2} \in L^{\infty}([0,T];L^{1}(\O));$$
$$\rho \in L^{\infty}([0,T];L^{\gamma}(\O));$$
$$\nabla \d \in L^{\infty}([0,T];L^{2}(\O));$$
$$F(\d)\in L^{\infty}([0,T];L^{1}(\O));$$
$$\nabla\u \in L^{2}([0,T];L^{2}(\O));$$
and also\begin{equation}\label{aaa}\Delta \d-f(\d)\in L^{2}([0,T];L^{2}(\O)).\end{equation}
Although the above estimates will play very important roles in proving of our main existence theorem, they cannot provide sufficient regularity for the direction field $\d$ to control the strongly nonlinear terms containing $\nabla\d$. To overcome this difficulty, we need the following lemma (see \cite{FRS}):

\begin{Lemma}\label{Lemma1}
If there exists a constant $C_{0}>0$ such that  $ \d\cdot f(\d)\geq 0 \text{ for all } |\d|\geq C_{0}>0, $ then $\d\in L^{\infty}((0,T)\times\O)$, $\nabla\d \in L^{4}((0,T)\times \O).$ \end{Lemma}

\begin{proof}
On one hand, if $|\d|<C_{0}$,  we have $$\d \in L^{\infty}((0,T)\times\O).$$
On the other hand, if $|\d|\ge C_{0}$,  taking the scalar product of equation \eqref{I1c} with $\d$ yields
\begin{equation*}
\partial_{t}|\d|^{2}-\Delta|\d|^{2}+\u\cdot\nabla|\d|^{2}+2\d\cdot f(\d)+2|\nabla\d|^{2}=0,
\end{equation*}
which implies
\begin{equation*}
\partial_{t}|\d|^{2}-\Delta|\d|^{2}+\u\cdot\nabla|\d|^{2}\leq0.
\end{equation*}
Using the maximum principle for $|\d|^{2}$ to obtain
\begin{equation}
\label{EI5}\d \in L^{\infty}((0,T)\times\O).
\end{equation}
Using \eqref{aaa}, \eqref{EI5}, smoothness of $f$, and together with elliptic estimate, we get
\begin{equation*}
\d\in L^{2}(0,T;H^{2}(\O)).
\end{equation*}
Using the Gagliardo-Nirenberg inequality, for some constant $C>0$,
\begin{equation*}
\|\nabla\d\|_{L^{4}(\O)}\leq C\|\Delta\d\|_{L^{2}(\O)}^{\frac{1}{2}}\|\d\|_{L^{\infty}(\O)}^{\frac{1}{2}}+C\|\d\|_{L^{\infty}(\O)},
\end{equation*}
which means
\begin{equation*}
\nabla\d \in L^{4}((0,T)\times\O).
\end{equation*}
The proof is complete.
\end{proof}

Through our paper, we will use $C$ to denote a generic positive constant, $\mathcal{D}$ to denote $C_{0}^{\infty}$,  and $\mathcal{D'}$ to denote the sense of distributions.
To introduce the finite energy weak solution $(\rho,\u,\d)$, we also need to take a differentiable function $b$, and multiply \eqref{I1a} by $b^{'}(\rho)$ to get the renormalized form:
\begin{equation}\label{EI7}
b(\rho)_{t}+\Dv(b(\rho)\u)+(b^{'}(\rho)\rho-b(\rho))\Dv \u =0. \end{equation}
We define the finite energy weak solution $(\rho,\u,\d)$ to the initial-boundary value problem \eqref{I1}-\eqref{I3} in the following sense: for any $T>0$,
\begin{itemize}%\addtolength{\itemsep}{-0.2\baselineskip}
\item  $\rho\geq 0, \quad \rho \in L^{\infty}([0,T];L^{\gamma}(\O)),\quad \u \in L^{2}([0,T];W^{1,2}_{0}(\O)),$
           $$\d \in L^{\infty}((0,T)\times\O) \cap L^{\infty}([0,T];H^{1}(\O))\cap L^{2}([0,T]; H^{2}(\O)),$$
          with $(\rho, \rho\u, \d)(0,x)=(\rho_{0}(x), \m_{0}(x), \d_{0}(x))$ for $x\in\O$;
\item The equations \eqref{I1} hold in $D^{'}((0,T)\times\O)$,  and \eqref{I1a} holds in $D^{'}((0,T)\times\R^{3})$ provided $\rho,\u$ are prolonged to be zero on $\R^{3}\setminus\O;$
\item \eqref{EI7} holds in $D^{'}((0,T)\times\O),$ for any $b\in C^{1}(\R^+)$ such that
\begin{equation}\label {EI8}b^{'}(z)=0 \text{ for all } z \in \R^+ \text{ large enough}, \text{ say } z \geq M, \end{equation}
where the constant $M$ may vary for different function $b$;
\item The energy inequality
$$E(t)+\int_0^t\int_{\O}\left(\mu|\nabla\u|^{2}dx+\lambda|\Delta\d-f(\d)|^{2}\right)dx ds\le E(0)$$
holds for almost every $t\in[0,T]$.
 \end{itemize}

\begin{Remark}
It's possible to deduce that \eqref{EI7} will hold for any $b \in C^{1}(0,\infty)\cap C[0,\infty)$ satisfying the following conditions
\begin{equation}\label{EI9}|b^{'}(z)|\leq c(z^{\a}+z^{\frac{\gamma}{2}})\text{ for all } z>0 \text{ and a certain }\a\in(0,\frac{\gamma}{2}) \end{equation} provided $(\rho,\u,\d)$ is a finite energy weak solution in the sense of the above definition (see details in \cite{FNP}).\end{Remark}

Now, our main result on the existence of finite energy weak solutions reads as follows:

\begin{Theorem}\label{T1}
Assume that $\O \subset \R^{3}$ is a bounded domain  of the class $C^{2+\nu}$, $\nu >0$,
and $\gamma>\frac{3}{2}$.
If there exists a $C_{0}>0$, such that  $ \d\cdot f(\d)\geq 0 \text{ for all } |\d|\geq C_{0}>0.$
 Then for any given $T>0$, the initial-boundary conditions \eqref{I1}-\eqref{I3} has a finite weak energy solution $(\rho,\u,\d)$ on $(0,T)\times \O.$
\end{Theorem}

%\begin{Remark}It's similar to prove the same conclusion in any bounded domain of dimensions two for any $\gamma>1.$    \end{Remark}

\begin{Remark} The typical example
$$F(\d)=\frac{1}{4\sigma_0^{2}}(|\d|^{2}-1)^{2}, \quad  f(\d)=\frac{1}{2\sigma_0^{2}}(|\d|^{2}-1)\d,$$
 satisfies the assumption of Theorem \ref{T1} with $C_0=1$. Thus the theorem holds for the typical case. \end{Remark}

Motivated by \cite{FP} and \cite {HW}, we establish the following result on the large-time behavior of the weak solutions to the problem \eqref{I1}-\eqref{I3}:
\begin{Theorem}\label{T2}
Assume that $(\rho,\u,\d)$ is the finite energy weak solution to \eqref{I1}-\eqref{I3} given in Theorem \ref{T1}, then there exists a stationary state of the density $\rho_{s}$ which is a positive constant, a stationary state of velocity $\u_{s}=0$, and a stationary state of direction field $\d_{s}$ such that, as $t\to\infty,$
\begin{equation}\label{Tlar1}\rho(t,x)\to \rho_{s} \text{ strongly in } L^{\gamma}(\O),\end{equation}
\begin{equation}\u(t,x)\to\u_{s}=0 \text{ strongly in } L^{2}(\O),\end{equation}
\begin{equation}\d(t,x)\to \d_{s} \text{ strongly in } H^{1}(\O),\end{equation}
where $\d_{s}$ solves the  equation \begin{equation}\label{equd1}\D \d_{s}=f(\d_{s})\end{equation} with the boundary condition\begin{equation}\label{equd2}\d_{s}|_{\partial\O}=\d_{0};\end{equation}
and $\rho_{s}$ satisfies the  following relation:
 \begin{equation}\label{equdensity}\nabla\rho_{s}^{\gamma}=-\lambda\Dv\left(\nabla\d_{s}\odot\nabla\d_{s}-(\frac{1}{2}|\nabla\d_{s}|^{2}+F(\d_{s}))I_{3}\right).\end{equation}
\end{Theorem}

\bigskip

\begin{Remark} The existence and uniqueness of \eqref{equd1} and \eqref{equd2} can be guaranteed from the elliptic theory,
and $\d_{s}\in C^{2}(\O)\cap C^{1}(\overline{\O})$ by the standard elliptic estimates (see \cite{GT}).  \end{Remark}

\begin{Remark} Denote $H=\rho_{s}^{\gamma}-\lambda F(\d_{s})$,
then equation \eqref{equdensity} can be rewritten as
\begin{equation}\label{asym}\nabla H =-\lambda\nabla\d_{s}\cdot\D\d_{s}.\end{equation}
By \eqref{equd1} and $f(\d_{s})=\nabla_{\d_{s}}F(\d_{s})$, we can deduce that $$\nabla(H+\lambda F(\d_{s}))=0,$$
that is, $$\nabla\rho_{s}^{\gamma}=0.$$
%Thus,  $$H+\lambda F(\d_{s})=c,$$
%for a constant $c$, and we obtain $\rho_{s}=c^{\frac{1}{\gamma}},$ where $c>0$ is uniquely determined by the mass $m(\rho_{s})=\int_{\O}\rho_{s}dx=\int_\O\rho_0dx,$ i.e., $ \rho_{s}=\frac{1}{|\O|}\int_{\O}\rho_{0}dx.$
By the way, in general, we cannot solve the equation \eqref{asym} without the condition $f(\d_{s})=\nabla_{\d_{s}}F(\d_{s}).$\end{Remark}

\begin{Remark} Our asymptotic equations \eqref{equd1}-\eqref{equdensity} for the compressible flow of liquid crystals  as $t\to\infty$ are similar to those for the incompressible flow of liquid crystals  obtained by Lin-Liu (see \cite{LL}). In particular, the asymptotic equations  \eqref{equd1}, \eqref{equd2},  and \eqref{asym} share the same form with those in \cite{LL}. \end{Remark}

The proof of Theorem \ref{T1} is based on the following approximation scheme:
\begin{subequations}\label{EI00}
\begin{align}&\rho_{t}+\Dv(\rho\u)=\varepsilon\Delta\rho,\label{EI00a}\\
&(\rho\u)_{t}+\Dv(\rho\u\otimes\u)+\nabla P(\rho)+\delta\nabla \rho^{\beta}+\varepsilon \nabla\u\cdot\nabla \rho \notag\\
&\qquad\qquad\qquad\qquad
 = \mu\Delta\u-\lambda\Dv\left(\nabla\d\odot\nabla\d-(\frac{1}{2}|\nabla\d|^{2}+F(\d)) I_{3}\right),\label{EI00b}\\
& \d_{t}+\u\cdot\nabla\d=\Delta\d-f(\d),\label{EI00c}
\end{align}\end{subequations}
with appropriate initial-boundary conditions.
Following the approach of Feireisl \cite{F, FNP},  we shall obtain the solution of \eqref{I1} when $\varepsilon\to 0$ and $ \delta\to 0$ in \eqref{EI00}. We can solve equation \eqref{EI00a} provided $\u$ is given. Indeed, we can obtain the existence by using classical theory of parabolic equation and overcome the difficulty of vacuum. Next we can also solve equation \eqref{EI00c} when $\u$ is fixed. By a direct application of the Schauder fixed point theorem, we can establish the local existence of $\u$,  and then extend this local solution to the whole time interval. Note that the addition of the extra term $\varepsilon\nabla\u\cdot\nabla\rho$ is necessary for keeping the energy conservation. The last step is to let $\varepsilon\to 0$ and $ \delta\to 0$ to recover the original system.  We remark that the strongly nonlinear terms containing $\nabla\d$ can be controlled by the sufficiently strong estimate about $\nabla\d$ obtained from the Gagliardo-Nirenberg inequality. In order to control the possible oscillations of the density $\rho$, we adopt the methods in Lions \cite{L} and Feireisl \cite{F, FNP} which is based on the celebrated weak continuity of the effective viscous flux $P-\mu \Dv \u$.  We refer the readers to Lions \cite{L}, Feireisl \cite{F,FNP}, and Hu-Wang \cite{HW} for discussions on the effective viscous flux.
\bigskip

%%%%%%%%%%%%%%%%%%%%%%%%%%
\section{The Solvability of the Direction Vector}

To solve the approximation system \eqref{EI00} by the Faedo-Galerkin method, we need to show that the following system can be uniquely solved in terms of $\u$:
\begin{subequations}\label{D1}\begin{align}&\d_{t}+\u\cdot\nabla\d=\Delta\d-f(\d),\label{D1a}\\
&\d|_{t=0}=\d_{0},\quad \d|_{\partial\O}=\d_{0},\label{D1b} \end{align}\end{subequations}
which can be achieved by the two lemmas below.

\begin{Lemma} \label{Le2} If $\u \in C([0,T];C^{2}_{0}(\bar{\O},\R^{3})),$ then there exists at most one function  $$ \d \in L^{2}(0,T;H^{2}_{0}(\O))\cap L^{\infty}([0,T];H^{1}(\O))$$
which solves \eqref{D1} in the weak sense on $\O\times(0,T)$,  and satisfies the initial and boundary  conditions in the sense of traces.\end{Lemma}
\begin{proof} Let $\d_{1},\d_{2}$ be two solutions of \eqref{D1} with the same data, then we have
\begin{equation}\label{D2}(\d_{1}-\d_{2})_{t}+\u\cdot\nabla(\d_{1}-\d_{2})=\Delta(\d_{1}-\d_{2})-(f(\d_{1})-f(\d_{2})).\end{equation}
Multiplying \eqref{D2} by $\Delta(\d_{1}-\d_{2}),$ integrating it over $\O,$ and using integration  by parts and the Cauchy-Schwarz inequality, we obtain
\begin{equation}\begin{split}&\partial_{t}\int_{\O}|\nabla(\d_{1}-\d_{2})|^{2}dx+2\int_{\O}|\Delta(\d_{1}-\d_{2})|^{2}dx\\&=2\int_{\O}(\nabla(\d_{1}-\d_{2}))^{\top}\cdot(\Delta(\d_{1}-\d_{2}))\cdot\u dx +2\int_{\O}(f(\d_{1})-f(\d_{2}))(\Delta(\d_{1}-\d_{2}))dx
\\& \leq C\int_{\O}|\nabla(\d_{1}-\d_{2})|^{2}dx+\int_{\O}|\Delta(\d_{1}-\d_{2})|^{2}dx,\end{split}\end{equation}
where we used the fact that $f$ is smooth. Then
\begin{equation}\label{D3}\partial_{t}\int_{\O}|\nabla(\d_{1}-\d_{2})|^{2}dx+\int_{\O}|\Delta(\d_{1}-\d_{2})|^{2}dx\leq C\int_{\O}|\nabla(\d_{1}-\d_{2})|^{2}dx,\end{equation}
and Lemma \ref{Le2} follows  from Gr\"{o}nwall's inequality, the above inequality,  together with Lemma \ref{Lemma1}.
\end{proof}

\begin{Lemma} \label{Le3} Let $\O \subset \R^{3}$ be a bounded domain of class $C^{2+\nu},\quad \nu>0.$ Assume that $\u \in C([0,T];C^{2}_{0}(\bar{\O},\R^{3}))$ is a given velocity field. Then the solution operator
$$\u\longmapsto \d[\u]$$ assigns to $\u \in C([0,T];C^{2}_{0}(\bar{\O};\R^{3}))$ the unique solution $\d$ of \eqref{D1}. Moreover, the operator $\u \longmapsto \d[\u]$ maps bounded sets in $C([0,T];C^{2}_{0}(\bar{\O};\R^{3}))$ into bounded subsets of $$Y:=L^{2}([0,T];H^{2}_{0}(\O))\cap L^{\infty}([0,T];H^{1}(\O)),$$
and the mapping $$\u \in C([0,T];C^{2}_{0}(\bar{\O};\R^{3}))\longmapsto \d\in Y$$ is continuous on any bounded subsets of $C([0,T];C^{2}_{0}(\bar{\O};\R^{3})).$
\end{Lemma}

\begin{proof} The uniqueness of the solution to \eqref{D1} is a consequence of Lemma \ref{Le3}, and the existence of a solution can be guaranteed by the standard parabolic equation theory. By \eqref{D3}, we can conclude that the solution operator $\u\longmapsto \d(\u)$ maps bounded sets in $C([0,T];C^{2}_{0}(\bar{\O};\R^{3}))$ into bounded subsets of the set $Y$.
%=L^{2}([0,T];H^{2}_{0}(\O))\cap L^{\infty}([0,T];H^{1}(\O)).$
Our next step is to show that the solution operator is continuous from any bounded subset of $C([0,T];C^{2}_{0}(\bar{\O}))$ to $Y$.
Let $\{\u_{n}\}_{n=1}^{\infty}$ be a bounded sequence in $C([0,T];C^{2}_{0}(\bar{\O}))$, that is to say, $\u_{n} \in B(0,R)\subset C([0,T];C^{2}_{0}(\bar{\O}))$ for some $R>0$, and
$$\u_{n}\to\u \text{ in } C([0,T];C^{2}_{0}(\bar{\O}))\quad \text{ as } n\to \infty.$$
Here, we denote $\d[\u]=\d$, and $\d[\u_{n}]=\d_{n}$,  so we have
\begin{equation}\begin{split}&\partial_{t}\int_{\O}\frac{1}{2}|\nabla(\d_{n}-\d)|^{2}dx+\int_{\O}|\Delta(\d_{n}-\d)|^{2}dx\\
&=\int_{\O}(\u\cdot\nabla\d-\u_{n}\cdot\nabla\d_{n})(\nabla(\d_{n}-\d))dx+\int_{\O}(f(\d)-f(\d_{n}))\cdot(\Delta(\d_{n}-\d))dx
\\& \leq \int_{\O}(|\u-\u_{n}|\cdot|\nabla\d|+|\u_{n}||\nabla(\d-\d_{n})|)|\Delta(\d_{n}-\d)|dx+C\int_{\O}|\nabla(\d_{n}-\d)|^{2}dx
\\& \leq \|\u_{n}-\u\|_{L^{\infty}}\|\nabla\d\|_{L^{2}}^{2}+C\|\nabla(\d-\d_{n})\|_{L^{2}}^{2}+\frac{1}{2}\int_{\O}|\Delta(\d_{n}-\d)|^{2}dx
\\& \leq C\|\u_{n}-\u\|_{L^{\infty}}+\frac{1}{2}\|\nabla(\d-\d_{n})\|_{L^{2}}^{2},\end{split}\end{equation}
where we used facts that $\d_{n}$ is bounded in $Y$ and $ f $ is smooth. This implies that
\begin{equation}\label{D4}\begin{split}
&\frac{1}{2}\partial_{t}\int_{\O}|\nabla(\d_{n}-\d)|^{2}dx+\frac{1}{2}\int_{\O}|\Delta(\d_{n}-\d)|^{2}dx
\\ & \leq C\|\u_{n}-\u\|_{L^{\infty}}+C\|\nabla(\d_{n}-\d)\|_{L^{2}}^{2}.
\end{split}
\end{equation}
Integrating \eqref{D4} over time $t \in (0,T)$, and then taking the upper limit over $n $ on the both sides, we get, noting that $\u_{n}\to\u \text{ in } C([0,T];C^{2}_{0}(\bar{\O});\R^{3}),$
\begin{equation}\begin{split}\label{D5}&\frac{1}{2}\lim_{n}\sup\int_{\O}|\nabla(\d_{n}-\d)|^{2}dx+\frac{1}{2}\lim_{n}\sup\int_{0}^{T}\int_{\O}|\Delta(\d_{n}-\d)|^{2}dxdt\\&\leq C \lim_{n}\sup\int_{0}^{T}\|\nabla(\d_{n}-\d)\|_{L^{2}}^{2}dt
\\ &\leq C\int_{0}^{T}\lim_{n}\sup\|\nabla(\d_{n}-\d)\|_{L^{2}}^{2}dt,\end{split}\end{equation}
thus, using Gr\"{o}nwall's inequality to \eqref{D5} and noting that $\d_{n},\,\d$ share the same initial data, we have
$$\lim_{n}\sup\int_{\O}|\nabla(\d_{n}-\d)|^{2}dx=0,$$
which means, from \eqref{D5} again,
$$\lim_{n}\sup\int_{0}^{T}\int_{\O}|\Delta(\d_{n}-\d)|^{2}dxdt=0.$$
Thus, we obtain
$$\d_{n}\to\d \text{ in } Y.$$
 This completes the proof of the continuity of the solution operator.
\end{proof}

\bigskip

%%%%%%%%%%%%%%%%%%%%%
\section{The Faedo-Galerkin Approximation Scheme}
In this section, we establish the existence of solution to the following approximation scheme:
\begin{subequations}\label{App1}
\begin{align}
&\rho_{t}+\Dv(\rho\u)=\varepsilon\Delta\rho\label{App1a},\\
&(\rho\u)_{t}+\Dv(\rho\u\otimes\u)+\nabla P(\rho)+\delta\nabla \rho^{\beta}+\varepsilon \nabla\u\cdot\nabla \rho \notag\\
&\qquad\qquad\qquad\qquad
 = \mu\Delta\u-\lambda\Dv\left(\nabla\d\odot\nabla\d-(\frac{1}{2}|\nabla\d|^{2}+F(\d)) I_{3}\right),\label{App1b}\\
& \d_{t}+\u\cdot\nabla\d=\Delta\d-f(\d),\label{App1c}
\end{align}\end{subequations}
with boundary conditions
\begin{subequations}\begin{align}
&\nabla\rho\cdot\nu|_{\partial\O}=0,\\
&\d|_{\partial\O}=\d_{0},\\
&\u|_{\partial\O}=0,
\end{align}\end{subequations}
together with modified initial data
\begin{subequations}\label{mid1}
\begin{align}
&\rho|_{t=0}=\rho_{0,\delta}(x),\\
&\rho\u|_{t=0}=\m_{0,\delta}(x),\\
&\d|_{t=0}=\d_{0}(x).
\end{align}\end{subequations}
Here the initial data $\rho_{0,\delta}(x)\in C^{3}(\overline{\O})$  satisfies the following conditions:
\begin{equation}
\label{App2a} 0<\delta\leq\rho_{0,\delta}(x)\leq \delta^{-\frac{1}{2\beta}},
\end{equation}
and \begin{equation}\label{App2b}
\rho_{0,\delta}(x)\to \rho \text{ in } L^{\gamma}(\O),\quad |\{\rho_{0,\delta}<\rho_{0}\}|\to 0 \quad \text{ as } \delta\to 0.
\end{equation}
Moreover,
\begin{equation}
\label{mid2}\m_{0,\delta}(x)=\begin{cases}\m_{0}\quad \text{ if } \rho_{0,\delta}(x) \geq \rho_{0}(x),\\
0 \quad\text{ if } \rho_{0,\delta}(x)<\rho_{0}(x).
\end{cases}\end{equation}

The density $\rho=\rho[\u]$ is determined uniquely as the solution of the following Neumann initial-boundary value problem (see Lemmas 2.1 and 2.2 of \cite{FNP}):
\begin{subequations}\label{App3}\begin{align}
&\rho_{t}+\Dv(\rho\u)=\varepsilon\Delta\rho ,\\
& \nabla\rho\cdot\nu|_{\partial\O}=0,\\
& \rho|_{t=0}=\rho_{0,\delta}(x),
\end{align}\end{subequations}
%and we can obtain the following bounds for $\rho$:
%\begin{equation}\label{low bdd}
%\begin{split}
%&\left(\inf_{\O}\rho_{0,\delta}\right)exp\left(-\int_{0}^{T}\|\Dv \u\|_{L^{\infty}}dt\right)\leq \rho(t,x)
%\\ &\leq \left(\sup_{\O}\rho_{0,\delta}\right)exp\left(\int_{0}^{T}\|\Dv \u\|_{L^{\infty}}dt\right),
%\end{split}
%\end{equation}
%for any $t\geq0$ and any $x \in \O$. This estimate  can help us  overcome difficulty from the vacuum.%

To solve \eqref{App1b} by a modified Faedo-Galerkin method, we need to introduce the finite-dimensional space
endowed with the $L^2$ Hilbert space structure:
$$X_{n}=span(\eta_{i})_{i=1}^{n}, \quad n\in\{1,2,3,\cdots\},$$
where the linearly independent functions $\eta_{i}\in \mathcal{D}(\O)^{3}$, $i=1,2, \dots$, form a  dense subset in $C_{0}^{2}(\overline{\O},\R^{3}).$
The approximate solution $\u_{n}$ should be given by the following form:
\begin{equation}\label{App4}\begin{split}&\int_{\O}\rho\u_{n}(\tau)\cdot\eta dx-\int_{\O}m_{0,\delta}\cdot\eta dx
\\&=\int_{0}^{\tau}\int_{\O}\left((\mu\Delta\u_{n}-\Dv(\rho\u_{n}\otimes\u_{n}))-\nabla(\rho^{\gamma}+\delta\rho^{\beta})-\varepsilon\nabla\rho\cdot\nabla\u_{n}\right)\cdot \eta dx dt\\
&-\int_{0}^{\tau}\int_{\O}\lambda\Dv(\nabla\d\odot\nabla\d-(\frac{1}{2}|\nabla\d|^{2}+F(\d))I_{3})\cdot\eta dxdt\end{split}\end{equation}
for any $t \in[0,T]$ and any $\eta\in X_{n}$, where $\varepsilon,\delta,\beta$ are fixed.
Due to Lemmas 2.1 and 2.2 of \cite{FNP} and our Lemmas \ref{Le2} and \ref{Le3}, the problem \eqref{D1}, \eqref{App3} and \eqref{App4} can be solved at least on a short time interval $(0,T_{n})$
with $T_{n}\leq T$ by a standard fixed point theorem on the Banach space $C([0,T],X_{n})$.
 We refer the readers to \cite{FNP} for more details. Thus we obtain a local solution $(\rho_{n},\u_{n},\d_{n})$ in time.

To obtain uniform bounds on $\u_{n}$, we derive an energy inequality similar to \eqref{EI2} as follows.
Taking $\eta=\u_{n}(t,x)$ with fixed $t$ in \eqref{App1} and repeating the procedure for a priori estimates in Section 2, we deduce a ``Kinetic energy equality":
\begin{equation}\begin{split}\label{App5}
&\partial_{t}\int_{\O}\left(\frac{1}{2}\rho_{n}|\u_{n}|^{2}
+\frac{1}{\gamma-1}\rho_{n}^{\gamma}+\frac{\delta}{\beta-1}\rho_{n}^{\beta}+\frac{\lambda}{2}|\nabla\d_{n}|^{2}+\lambda F(\d_{n})\right)dx
+\mu\int_{\O}|\nabla\u_{n}|^{2}dx
\\&+\lambda\int_{\O}|\Delta\d_{n}-f(\d_{n})|^{2}dx+\varepsilon\int_{\O}(\gamma\rho_{n}^{\gamma-2}
+\delta\beta\rho_{n}^{\beta-2})|\nabla\rho_{n}|^{2}dx = 0.
\end{split}\end{equation}
The uniform estimates obtained from \eqref{App5} furnish the possibility of repeating the above fixed point argument to extend the local solution
$\u_{n}$ to the whole time interval $[0,T].$ Then, by the solvability of equation \eqref{App3} and \eqref{D1}, we obtain the functions $(\rho_{n},\d_{n})$ on the whole time interval $[0,T].$

The next step in the proof of Theorem \ref{T1} is to pass the limit as $n\to \infty$ in the sequence of approximate solutions
$\{\rho_{n},\u_{n},\d_{n} \}$ obtained above. We observe that the terms related to $\u_{n}$ and $\rho_{n}$ can be treated similarly to \cite{FNP}. It remains to show the convergence of the terms related to $\d_{n}.$

By \eqref{App5}, smoothness of $f$, and elliptic estimates, we conclude
\begin{equation}
\label{+3}
\nabla\u_{n} \in L^{2}([0,T];L^{2}(\O)),
\end{equation}
\begin{equation}\label{+2}
\Delta\d_{n}-f(\d_{n}) \text{ is bounded in } L^{2}([0,T];L^{2}(\O)),
\end{equation}
and $$\d_{n}\in L^{\infty}([0,T];H^{1}_{0}(\O))\cap L^{2}([0,T]; H^{2}_{0}(\O)).$$
This yields that $$\Delta\d_{n}-f(\d_{n})\to \D\d-f(\d) \text{ weakly in }  L^{2}([0,T];L^{2}(\O)),$$
and  \begin{equation}\label{Dw}\d_{n}\to \d \text { weakly in }L^{\infty}([0,T];H^{1}_{0}(\O))\cap L^{2}([0,T]; H^{2}_{0}(\O)).\end{equation}
Using corollary 2.1 in \cite{F} and \eqref{App1c}, we can improve \eqref{Dw} as follows:
$$\d_{n}\to \d \quad  \text{ in } C([0,T];L^{2}_{weak}(\O)).$$
%From Lemma \ref{Lemma1}, we can conclude that $$\nabla\d_{n}\in L^{4}((0,T)\times\O).$$
%Multiplying $\d_{n}^{p}$ for any $p>1$ on both sides of \eqref{App1c},
%we have \begin{equation}\begin{split}&\int_{0}^{T}\int_{\O}\partial_{t}|\d_{n}|^{p+1}dxdt\\&\leq \int_{0}^{T}\int_{\O}|\d_{n}^{p}||\D\d_{n}-f(\d_{n})|dxdt+\int_{0}^{T}\int_{\O}|\u_{n}||\nabla\d_{n}||\d_{n}|dxdt\\&
%\leq \|\D\d_{n}-f(\d_{n})\|_{L^{2}(0,T;L^{2}(\O))}+C\|\u_{n}\|_{L^{2}(0,T;L^{2}(\O))}\|\nabla\d_{n}\|_{L^{2}(0,T;L^{2}(\O)},
%\end{split}\end{equation}
%where we used $\d_{n}\in L^{\infty}((0,T)\times\O)$.
%So we can conclude that $\partial_{t}\d_{n} \in L^{p+1}((0,T)\times\O).$
Next we need to rely on the following Aubin-Lions compactness lemma (see \cite{Lj}):

\begin{Lemma}\label{L+}
Let $X_{0}, X$ and $X_{1}$ be three Banach spaces with $X_{0}\subseteq X\subseteq X_{1}$. Suppose that $X_{0}$ is compactly embedded in $X$ and that $X$ is continuously embedded  in $X_{1}$; Suppose also that $X_{0}$ and $X_{1}$ are reflexive spaces. For $1 <p,q<\infty,$ let
$$W=\{u\in L^{p}([0,T];X_{0})|\frac{du}{dt} \in L^{q}([0,T];X_{1})\}.$$
Then the embedding of $W$ into $L^{p}([0,T];X)$ is also compact.
\end{Lemma}

We are now applying the Aubin-Lions lemma to obtain the convergence of $\d_{n}$ and $\nabla\d_{n}$.
%which means $$\partial_{t}\d_{n} \to \partial_{t}\d \text { weakly in } L^{p+1}((0,T)\times\O).$$
From Lemma \ref{Lemma1}, we have $$\d_{n}\in L^{\infty}((0,T)\times\O),$$
and \begin{equation}\label{+1}
\nabla\d_{n}\in L^{4}((0,T)\times\O).
\end{equation}
 %$$\d_{n}\in L^{\infty}((0,T)\times\O)\cap L^{4}((0,T),W^{1,4}(\O)).$$
%By the embedding  $W^{1,4}(\O)\hookrightarrow L^{\infty}(\O)$,  up to a subsequence if necessary,
%there exists a function $$\d \in L^{\infty}((0,T)\times\O)\cap L^{4}((0,T),W^{1,4}(\O))$$
%such that $\d_{n}\to \d$ strongly in $L^{\infty}(\O)$ almost everywhere $t \in [0,T].$

%Notice that $\d_{n}\in L^{\infty}((0,T)\times\O)$ and by \eqref{App5}, we have $\d_{n}\in L^{2}([0,T];H^{2}_{0}(\O)).$
%Thus,  $$\nabla\d_{n} \text { is bounded in } L^{\infty}([0,T];L^{2}(\O))\cap L^{2}([0,T];H^{1}_{0}(\O)).$$
%This implies that, by the embedding of $ H^{1}(\O)\hookrightarrow L^{2}(\O)$,  up to a subsequence if necessary, there exists a function $\nabla\d \in L^{\infty}([0,T];L^{2}(\O))\cap L^{2}([0,T];H^{1}(\O))$ such that
%$\nabla\d_{n}\to \nabla\d $ in $ L^{2}(\O)$ for almost everywhere $t \in [0,T].$
Using \eqref{App1c}, we have
\begin{equation*}
\begin{split}
\|\partial_{t}\d_{n}\|_{L^{2}(\O)}&\leq C\|\u_{n}\cdot\nabla\d_{n}\|_{L^{2}(\O)}+C\|\D\d_{n}-f(\d_{n})\|_{L^{2}(\O)}
\\& \leq C\|\u_{n}\|^{2}_{L^{4}(\O)}+C\|\nabla\d_{n}\|^{2}_{L^{4}(\O)}+C\|\D\d_{n}-f(\d_{n})\|_{L^{2}(\O)},
\\& \leq C\|\nabla \u_{n}\|^{2}_{L^{2}(\O)}+C\|\nabla\d_{n}\|^{2}_{L^{4}(\O)}+C\|\D\d_{n}-f(\d_{n})\|_{L^{2}(\O)},
\end{split}
\end{equation*}
where we used embedding inequality, the values of $C$ are variant. Thus, \eqref{+3}, \eqref{+2} and \eqref{+1} yield
\begin{equation*}
\|\partial_{t}\d_{n}\|_{L^{2}([0,T];L^{2}(\O))}\leq C.
\end{equation*}
Notice that $H^{2}\subset H^{1} \subset L^{2}$ and the injection $H^{2} \hookrightarrow H^{1}$ is compact, applying Lemma \ref{L+} we deduce that
the sequence $\{\d_{n}\}_{n=1}^{\infty}$ is  precompact in $L^{2}([0,T];H^{1}(\O)).$

Summing up the previous results, by taking a subsequence if necessary, we can assume that:
$$\d_{n}\to \d \quad  \text{ in } C([0,T];L^{2}_{weak}(\O)),$$
$$\d_{n}\to\d \text { weakly in } L^{2}(0,T;H^{2}(\O))\cap L^{\infty}(0,T;H^{1}(\O)),$$
$$\d_{n}\to\d \text { strongly in } L^{2}(0,T;H^{1}(\O)),$$
$$\nabla\d_{n}\to\nabla\d \text{ weakly in } L^{4}((0,T)\times\O),$$
$$\D\d_{n}-f(\d_{n})\to \D\d-f(\d) \text { weakly in } L^{2}(0,T;L^{2}(\O)),$$
$$F(\d_{n})\to F(\d) \text { strongly in } L ^{2}(0,T;H^{1}(\O)).$$
Now, we consider the convergence of the terms related to $\d_{n}$ and $\nabla\d_{n}$.
Let $\varphi$ be a test function, then
\begin{equation}\begin{split}\label{con1}&\int_{\O}(\nabla\d_{n}\odot\nabla\d_{n}-\nabla\d\odot\nabla\d)\cdot\nabla\varphi dxdt\\&\leq \int_{\O}(\nabla\d_{n}\odot\nabla\d_{n}-\nabla\d_{n}\odot\nabla\d)\nabla\varphi dxdt +\int_{\O}(\nabla\d_{n}\odot\nabla\d-\nabla\d\odot\nabla\d)\nabla\varphi dxdt\\&
\leq C\|\nabla\d_{n}\|_{L^{2}(\O)}\|\nabla\d_{n}-\nabla\d\|_{L^{2}(\O)}+C\|\nabla\d\|_{L^{2}(\O)}\|\nabla\d_{n}-\nabla\d\|_{L^{2}(\O)}\end{split}\end{equation}
By the strong convergence of $\nabla\d_{n}$ in $L^{2}(\O)$ and \eqref{con1}, we conclude that
$$\nabla\d_{n}\odot\nabla\d_{n}\to \nabla\d\odot\nabla\d \text{ in }\mathcal{D^{'}}(\O\times(0,T)).$$
Similarly,
$$\frac{1}{2}|\nabla\d_{n}|^{2}I_{3}\to \frac{1}{2}|\nabla\d|^{2}I_{3} \text{ in }\mathcal{D^{'}}(\O\times(0,T)),$$
and
$$\u_{n}\nabla\d_{n}\to\u\nabla\d \text{ in } \mathcal{D^{'}}(\O\times(0,T)),$$
where we used $$\u_{n}\to\u \text { weakly in } L^{2}([0,T];H^{1}_{0}(\O)).$$
Therefore, \eqref{D1} and \eqref{App4} hold at least in the sense of distribution. Moreover, by the uniform estimates on $\u, \d$ and \eqref{I1c}, we know that the map
$$t\to \int_{\O}\d_{n}(x,t)\varphi(x)dx \quad \text { for any } \varphi \in \mathcal{D}(\O),$$
is equi-continuous on $[0,T].$ By the Ascoli-Arzela Theorem, we know that
$$t\to \int_{\O}\d(x,t)\varphi(x)dx$$
is continuous for any $\varphi\in \mathcal{D}(\O).$ Thus, $\d$ satisfies the initial condition in \eqref{D1}.

Now we have the  existence of a global solution to \eqref{App1} as follows:
\begin{Proposition}\label{P1}
Assume that $\O \subset \R^{3}$ is a bounded domain  of the class $C^{2+\nu}$, $\nu >0$; and
 there exists a constant $C_{0}>0$, such that $ \d\cdot f(\d)\geq 0 \text{ for all } |\d|\geq C_{0}>0. $
Let $\varepsilon>0,\delta>0,$ and $\beta >\text{max}\{4,\gamma\}$ be fixed. Then for any given $T>0$, there is a solution $(\rho, \u,\d)$ to the initial-boundary value problem of \eqref{App1} in the following sense:

(1) The density $\rho$ is a nonnegative function such that
$$\rho \in L^{\gamma}([0,T];W^{2,r}(\O)),\quad \partial_{t}\rho \in L^{r}((0,T)\times\O),$$
for some $r >1$, the velocity $\u \in L^{2}([0,T];H^{1}_{0}(\O))$, and \eqref{App1a} holds almost everywhere on $(0,T)\times\O, $
and the initial and boundary data on $\rho$ are satisfied in the sense of traces. Moreover, the total mass is conserved, i.e.
$$\int_{\O}\rho(x,t)dx=\int_{\O}\rho_{\delta,0}dx,$$
for all $t \in [0,T];$ and the following inequalities hold
$$\delta\int_{0}^{T}\int_{\O}\rho^{\beta+1}dx dt\leq C(\varepsilon),$$
$$\varepsilon\int_{0}^{T}\int_{\O}|\nabla\rho|^{2}dxdt \leq C \text{ with } C \text{ independent of } \varepsilon.$$

(2) All quantities appearing in equation \eqref{App1b} are locally integrable, and the equation is satisfied in $\mathcal{D^{'}}(\O\times(0,T)).$ Moreover, $$\rho\u \in C([0,T];L^{^{\frac{2\gamma}{\gamma+1}}}_{weak}(\O)),$$ and $\rho\u$ satisfies the initial data.

(3) All terms in \eqref{App1c} are locally integrable on $\O\times(0,T)$. The direction $\d$ satisfies the equation \eqref{D1a} and the initial data \eqref{D1b} in the sense of distribution.

(4) The energy inequality
\begin{equation*}\begin{split}
&\partial_{t}\int_{\O}\left(\frac{1}{2}\rho|\u|^{2}
+\frac{1}{\gamma-1}\rho^{\gamma}+\frac{\delta}{\beta-1}\rho^{\beta}+\frac{\lambda}{2}|\nabla\d|^{2}+\lambda F(\d)\right)dx\\&+
\mu\int_{\O}|\nabla\u|^{2}dx+\lambda\int_{\O}|\D\d-f(\d)|^{2}dx\\&
\leq 0
\end{split}\end{equation*}
holds almost everywhere for $t \in[0,T]$.
\end{Proposition}

To complete our proof of the main theorem, we will take vanishing artificial viscosity and vanishing artificial pressure in the following sections.

\bigskip

%%%%%%%%%%%%%%%%%
\section{Vanishing Viscosity Limit}

 In this section, we will pass the limit as $\varepsilon\to 0$ in the family of approximate solutions $(\rho_{\varepsilon},\u_{\varepsilon},\d_{\varepsilon})$ obtained in Section 4. The estimates in Proposition \ref{P1}
are independent of $n$, and those estimates are still valid for $(\rho_{\varepsilon},\u_{\varepsilon},\d_{\varepsilon})$. But, we need to remark that $\rho_{\varepsilon}$ will lose some regularity when $\varepsilon \to 0$ because the term $ \varepsilon \Delta\rho_{\varepsilon} $ goes away.  The space $L^{\infty}(0,T;L^{1}(\O))$ is a non-reflexive space, and the artificial pressure is bounded only in space $L^{\infty}(0,T;L^{1}(\O))$ from the estimates of Proposition \ref{P1}. It is crucial to establish the strong compactness of the density $\rho_{\varepsilon}$ for passing the limits. To this end, we need to obtain better estimates on the artificial pressure.

\subsection{Uniform estimates of the density}% uniformly of viscosity}

We first introduce an operator
$$B:\left\{f\in L^{p}(\O):\int_{\O}fdx=0 \right\} \longmapsto[W^{1,p}_{0}(\O)]^{3}$$
which is a bounded linear operator satisfying
\begin{equation}\label{v1+}\|B[f]\|_{W^{1,p}_{0}(\O)}\leq c(p)\|f\|_{L^{p}(\O)}\quad \text { for any } 1<p<\infty,\end{equation}
where the function $W=B[f]\in \R^{3}$ solves the following equation:
$$\Dv W=f \text { in } \O,\quad W|_{\partial \O}=0.$$
 Moreover, if the function $f$ can be written in the form $ f =\Dv g $ for some $g \in L^{r}, $ and $g\cdot \nu|_{\partial{\O}}=0,$
then $$\|B[f]\|_{L^{r}(\O)}\leq c(r)\|g\|_{L^{r}(\O)}$$  for any $1<r<\infty.$
We refer the readers to \cite{F,FNP} for more background and discussion of the operator $B$.
Define the function:
$$\varphi(t,x)=\psi(t)B[\rho_{\varepsilon}-\widehat{\rho}],\quad \psi\in \mathcal{D}(0,T),\quad 0\leq \psi \leq 1,$$
where $$\widehat{\rho}=\frac{1}{|\O|}\int_{\O}\rho(t) dx.$$
Since $\rho_{\varepsilon}$ is a solution to \eqref{App1a}, by Proposition \ref{P2} and  $\beta>4$, we have
$$\rho_{\varepsilon}-\widehat{\rho}\in C([0,T],L^{4}(\O)).$$
Therefore, from \eqref{v1+}, we have $\varphi(t,x) \in C([0,T],W^{1,4}(\O)). $
In particular, $\varphi(t,x) \in C([0,T]\times\O)$ by the Sobolev embedding theorem. Consequently, $\varphi$ can be used as a test function for \eqref{App1b}.
After a little bit lengthy but straightforward computation, we obtain:
\begin{equation}\begin{split}
&\int_{0}^{T}\int_{\O}\psi(\rho_{\varepsilon}^{\gamma+1}+\delta\rho_{\varepsilon}^{\delta+1})dxdt\\
&=\widehat{\rho}\int_{0}^{T}\int_{\O}\psi(\rho_{\varepsilon}^{\gamma}+\delta\rho_{\varepsilon}^{\beta})dxdt+\int_{0}^{T}\int_{\O}\psi\rho_{\varepsilon}\u_{\varepsilon}B[\rho_{\varepsilon}-\widehat{\rho}]dxdt\\
&\quad +\mu\int_{0}^{T}\int_{\O}\psi\nabla \u_{\varepsilon}\nabla B[\rho_{\varepsilon}-\widehat{\rho}]dxdt\\
&\quad -\int_{0}^{T}\int_{\O}\psi\rho_{\varepsilon}\u_{\varepsilon}\otimes\u_{\varepsilon}\nabla B[\rho_{\varepsilon}-\widehat{\rho}]dxdt\\
&\quad -\varepsilon\int_{0}^{T}\int_{\O}\psi\rho_{\varepsilon}\u_{\varepsilon}B[\Delta\rho_{\varepsilon}]dxdt\\
&\quad -\int_{0}^{T}\int_{\O}\psi\rho_{\varepsilon}\u_{\varepsilon}B[\Dv(\rho_{\varepsilon}\u_{\varepsilon})]dxdt\\
&\quad +\varepsilon\int_{0}^{T}\int_{\O}\nabla\u_{\varepsilon}\nabla\rho_{\varepsilon}B[\rho_{\varepsilon}-\widehat{\rho}]dxdt\\
&\quad -\lambda\int_{0}^{T}\int_{\O}\left(\nabla\d_{\varepsilon}\otimes\nabla\d_{\varepsilon}-(\frac{|\nabla\d_{\varepsilon}|^{2}}{2}+F(\d))I_{3}\right)\psi \nabla B[\rho_{\varepsilon}-\widehat{\rho}]dxdt\\
&=\sum_{j=1}^{7}I_{j}.
\end{split}\end{equation}
To achieve our  lemma below, we need to estimate that the terms $I_{1}-I_{7}$ are bounded. We can treat the terms related to $\rho_{\varepsilon},\u_{\varepsilon}$ similar to \cite{FNP}.
It remains to estimate the term $I_{7}$. Indeed,

\begin{equation}\begin{split}
&\left|I_{7}\right|=
\left|\lambda\int_{0}^{T}\int_{\O}\left(\nabla\d_{\varepsilon}\otimes\nabla\d_{\varepsilon}-(\frac{|\nabla\d_{\varepsilon}|^{2}}{2}+F(\d))I_{3}\right)\psi \nabla B[\rho_{\varepsilon}-\widehat{\rho}]dxdt\right|\\
&\leq C\lambda\int_{0}^{T} \|\nabla\d_{\varepsilon}\|_{L^{4}(\O)}^{2}\|B[\rho_{\varepsilon}-\widehat{\rho}]\|_{W^{1,2}(\O)}dt+C\int_{0}^{T}\|B[\rho_{\varepsilon}-\widehat{\rho}]\|_{W^{1,2}(\O)}dt
\\& \leq C,
\end{split}\end{equation}
where we used $$\|B[\rho_{\varepsilon}-\widehat{\rho}]\|_{W^{1,2}(\O)}\leq C_{0}\|\rho_{\varepsilon}-\widehat{\rho}\|_{L^{2}(\O)},$$
 $$\nabla\d_{\varepsilon} \in L^{4}([0,T]\times\O),$$
and  $\beta\geq 4.$
Consequently, we have proved the following result:
\begin{Lemma}\label{Le4}
Let $(\rho_{\varepsilon},\u_{\varepsilon},\d_{\varepsilon})$ be the solutions of the problem \eqref{App1} constructed in Proposition \ref{P1}, then
$$\|\rho_{\varepsilon}\|_{L^{\gamma+1}((0,T)\times\O)}+\|\rho_{\varepsilon}\|_{L^{\beta+1}((0,T)\times\O)} \leq C,$$
where $ C $ is independent of $\varepsilon$.
\end{Lemma}

\subsection{ The vanishing viscosity limit passage}
From the previous energy estimates, we have
$$\varepsilon\D\rho_{\varepsilon}\to 0 \quad \text{ in } L^{2}(0,T;W^{-1,2}(\O))$$
and $$\varepsilon\nabla\u_{\varepsilon}\nabla\rho_{\varepsilon}\to 0 \quad \text { in } L^{1}(0,T;L^{1}(\O))$$
as $\varepsilon\to 0.$

Due to the above estimates so far, we may now assume that
\begin{subequations}\label{v0+}\begin{align}
&\rho_{\varepsilon}\to\rho \text{ in } C([0,T],L^{\gamma}_{weak}(\O)),\\
&\u_{\varepsilon}\to \u \text{ weakly in } L^{2}(0,T;W^{1,2}_{0}(\O)),\\
&\rho_{\varepsilon}\u_{\varepsilon}\to \rho\u \text{ in } C([0,T],L^{\frac{2\gamma}{\gamma+1}}_{weak}(\O)).
\end{align}
\end{subequations}
Then we can pass the limits of the terms related to $\rho_{\varepsilon},\u_{\varepsilon}$ similarly to \cite{FNP}.
 It remains to show the convergence of  $\d_{\varepsilon}$. Following the same arguments of Section 4, by taking a subsequence if necessary, we can assume that:
\begin{subequations}\label{v0}\begin{align}
&\d_{\varepsilon}\to \d \quad  \text{ in } C([0,T];L^{2}_{weak}(\O))\label{v0a}
\\ & \d_{\varepsilon}\to\d \text { weakly in } L^{2}(0,T;H^{2}(\O))\cap L^{\infty}(0,T;H^{1}(\O)),\label{v0b}\\
&\d_{\varepsilon}\to\d \text { strongly in } L^{2}(0,T;H^{1}(\O)),\label{v0c}\\
&\nabla\d_{\varepsilon}\to\nabla\d \text{ weakly in } L^{4}((0,T)\times\O),\label{v0d}\\
&\D\d_{\varepsilon}-f(\d_{\varepsilon})\to \D\d-f(\d) \text { weakly in } L^{2}(0,T;L^{2}(\O)),\label{v0e}\\
&F(\d_{\varepsilon})\to F(\d) \text { strongly in } L ^{2}(0,T;H^{1}(\O)).\label{v0f}
\end{align}
\end{subequations}
Consequently, letting $\varepsilon\to 0$ and making use of \eqref{v0+} and \eqref{v0}, we conclude that the limit of $(\rho_{\varepsilon},\u_{\varepsilon},\d_{\varepsilon})$ satisfies the following system:
\begin{subequations}\label{V1}\begin{align}
&\rho_{t}+\Dv(\rho\u)=0,\label{V1a}\\
&(\rho\u)_{t}+\Dv(\rho\u\otimes\u)+\nabla \bar{P}=\mu\Delta \u-\lambda\Dv(\nabla\d\odot\nabla\d-(\frac{1}{2}|\nabla\d|^{2}+F(\d))I_{3}),\label{V1b}\\
&\d_{t}+\u\cdot\nabla\d=\Delta\d-f(\d)\label{V1c}
\end{align}\end{subequations}
where $\bar{P}=\overline{a\rho_{\varepsilon}^{\gamma}+\delta\rho_{\varepsilon}^{\beta}},$ here $\overline K(x)$ stands for a weak limit of $\{K_{\varepsilon}\}$. %_{n=1}^{\infty}.$

%Next we need to show  the strong convergence of $\rho_{\varepsilon}$. To this end, we adopt techniques in \cite{FNP} to define a defect measure:
%\begin{equation*}
%dft[\rho_{\varepsilon}-\rho]=\int_{\O}\overline{\rho \log\rho(t)}-\rho\log\rho(t)dx.
%\end{equation*}

%In the spirit of \cite{FNP},
\subsection{The strong convergence of the density}

 We observe that $\rho_{\varepsilon}, \u_{\varepsilon}$ is a strong solution of parabolic equation \eqref{App1a}, then the renormalized form can be written as

\begin{equation}\begin{split}\label{v2+}
&\partial_{t}b(\rho_{\varepsilon})+\Dv(b(\rho_{\varepsilon})\u_{\varepsilon})
+(b^{'}(\rho_{\varepsilon})\rho_{\varepsilon}-b({\rho_{\varepsilon}}))\Dv\u_{\varepsilon}\\&=\varepsilon\Dv(\chi_{\O}\nabla b(\rho_{\varepsilon}))-\varepsilon \chi_{\O}b^{''}(\rho_{\varepsilon})|\nabla\rho_{\varepsilon}|^{2}
\end{split}\end{equation}
in $D^{'}((0,T)\times \R^{3}),$ with $b\in C^{2}[0,\infty),\quad b(0)=0,$ and $b^{'},\quad b^{''}$ bounded functions and $b$ convex,
where $ \chi_{\O}$ is the characteristics function of $\O.$
By the virtue of \eqref{v2+} and the convexity of $b$, we have
\begin{equation*}
\int_{0}^{T}\int_{\O}\psi(b^{'}(\rho_{\varepsilon})\rho_{\varepsilon}-b(\rho_{\varepsilon})))\Dv\u_{\varepsilon}dxdt
\leq\int_{\O}b(\rho_{0,\delta})dx+\int_{0}^{T}\int_{\O}\partial_{t}\psi b(\rho_{\varepsilon})dxdt
\end{equation*}
for any $\psi \in C^{\infty}[0,T], \quad 0 \leq \psi \leq 1,\quad \psi(0)=1,\quad \psi(T)=0.$
Taking $b(z)=z \log z$ gives us the following estimate:
\begin{equation*}
\int_{0}^{T}\int_{\O}\psi\rho_{\varepsilon}\Dv \u_{\varepsilon}dxdt\leq\int_{\O}\rho_{0,\delta}\log(\rho_{0,\delta})dx
+\int_{0}^{T}\int_{\O}\partial_{t}\psi\rho_{\varepsilon}\log\rho_{\varepsilon}dxdt,
\end{equation*}
and letting $\varepsilon \to 0$ yields
\begin{equation*}
\label{V2''}
\int_{0}^{T}\int_{\O}\psi\overline{\rho\Dv\u}dxdt
\leq\int_{\O}\rho_{0,\delta}\log\rho_{0,\delta}dx+\int_{0}^{T}\int_{\O}\partial_{t}\psi\overline{\rho\log\rho}dxdt,
\end{equation*}
that is,
\begin{equation}
\label{V2}
\int_{0}^{T}\int_{\O}\overline{\rho\Dv \u}dxdt \leq\int_{\O}\rho_{0,\delta}\log\rho_{0,\delta}dx-\int_{\O}\overline{\rho\log\rho}(t)dx.
\end{equation}
Meanwhile, $(\rho, \u)$ satisfies
\begin{equation}
\label{V2similar}
\partial_{t}b(\rho)+\Dv(b(\rho)\u)+(b^{'}(\rho)\rho-b(\rho))\Dv\u=0.
\end{equation}
Using \eqref{V2similar} and $b(z)=z \log z$, we deduce the following inequality:
\begin{equation}
\label{V2+}
\int_{0}^{T}\int_{\O}\rho\Dv u dxdt \leq \int_{\O}\rho_{0,\delta}\log\rho_{0,\delta}dx-\int_{\O}\rho\log\rho(t)dx.
\end{equation}
From \eqref{V2+} and \eqref{V2}, we deduce that
\begin{equation}
\label{V2'}
\int_{\O}\overline{\rho\log\rho}-\rho\log(\rho)(\tau)dx\leq\int_{0}^{T}\int_{\O}\rho\Dv\u-\overline{\rho\Dv \u}dxdt
\end{equation}
for a.e. $\tau \in[0,T].$

To obtain the strong convergence of density $\rho_{\varepsilon}$, the crucial point is to get the weak continuity of the viscous pressure, namely:
\begin{Lemma}\label{Le6+}
Let $(\rho_{\varepsilon},\u_{\varepsilon})$ be the sequence of approximate solutions constructed in Proposition \ref{P1}, then
\begin{equation*}\begin{split}
&\lim_{\varepsilon\to 0^{+}}\int_{0}^{T}\int_{\O}\psi\eta(a\rho_{\varepsilon}^{\gamma}+
\delta\rho_{\varepsilon}^{\beta}-\mu\Dv\u_{\varepsilon})\rho_{\varepsilon}dxdt\\
&=\int_{0}^{T}\int_{\O}\psi\eta(\bar{P}-\mu\Dv\u)\rho dxdt\quad \text{ for any }\psi \in \mathcal{D}(0,T),\quad \eta \in \mathcal{D}(\O),
\end{split}\end{equation*}
where $\bar{P}=\overline{a\rho^{\gamma}+\delta\rho^{\beta}}.$
\end{Lemma}

\begin{proof}
We need to introduce a new operator $$A_{i}=\Delta^{-1}(\partial_{x_{i}}v), \; i=1,2,3,$$
where $\Delta^{-1}$ stands for the inverse of the Laplace operator on $\R^{3}.$ To be more specific, $A_{i}$ can be expressed by their Fourier symbol
$$A_{i}(\cdot)=\mathcal{F}^{-1}(\frac{-i\xi_{i}}{|\xi|^{2}}\mathcal{F}(\cdot)),\quad i=1,2,3,$$
with the following properties (see \cite{FNP}):
$$\|A_{i}v\|_{W^{1,s}(\O)}\leq c(s,\O)\|v\|_{L^{s}(\R^{3})},\quad 1<s<\infty, $$
$$\|A_{i}v\|_{L^{q}(\O)}\leq c(q,s,\O)\|v\|_{L^{s}(\R^{3})},\quad q<\infty,
\quad \text{ provided } \frac{1}{q}\geq\frac{1}{s}-\frac{1}{3},$$
and
$$\|A_{i}v\|_{L^{\infty}(\O)}\leq c(s,\O)\|v\|_{L^{s}(\R^{3})}\quad \text { if }s>3.$$
%It is easy to get $$\partial_{x_{i}}A_{i}(v)=v.$$
%Prolonging $\rho_{\varepsilon}$ to be zero outside $\O$,
%$$\varphi(t,x)=\psi(t)\eta(x)A_{i}[\rho_{\varepsilon}],\quad \psi\in D(0,T),\quad \eta \in D(\O),\quad i=1,2,3.$$
%By \eqref{App1a} and Proposition \ref{P1},
%$$\rho_{\varepsilon} \in C([0,T];L^{4}(\O)),$$
%and using $$\|A_{i}v\|_{W^{1,s}(\O)}\leq c(s,\O)\|v\|_{L^{s}(R^{3})},\quad 1<s<\infty, $$  we deduce that
%$$\varphi(t,x) \in C([0,T];W^{1,4}(\O)).$$
%By the embedding theorem, we conclude that $\varphi(t,x)\in C([0,T]\times\O)$,
%so it can be used to

Next, we use the quantities$$\varphi(t,x)=\psi(t)\eta(x)A_{i}[\rho_{\varepsilon}],\quad \psi\in \mathcal{D}(0,T),\quad \eta \in \mathcal{D}(\O),\quad i=1,2,3,$$ as a test function for \eqref{App1b} to obtain
\begin{equation}\label{V3}\begin{split}
&\int_{0}^{T}\int_{\O}\varphi\eta((\rho_{\varepsilon}^{\gamma}+\delta\rho_{\varepsilon}^{\beta})-\mu\Dv \u_{\varepsilon})\rho_{\varepsilon}dxdt\\
&=\mu\int_{0}^{T}\int_{\O}\psi\nabla\u_{\varepsilon}\nabla\eta A[\rho_{\varepsilon}]dxdt-\int_{0}^{T}\int_{\O}\psi(\rho_{\varepsilon}^{\gamma}+\delta\rho_{\varepsilon}^{\beta})\nabla\eta A[\rho_{\varepsilon}]dxdt\\
&\quad -\int_{0}^{T}\int_{\O}\psi\rho_{\varepsilon}\u_{\varepsilon}\otimes\u_{\varepsilon}\nabla\eta A[\rho_{\varepsilon}]dxdt-\int_{0}^{T}\int_{\O}\psi_{t}\eta\rho_{\varepsilon}\u_{\varepsilon}A[\rho_{\varepsilon}]dxdt\\
&\quad -\varepsilon\int_{0}^{T}\int_{\O}\psi\eta\rho_{\varepsilon}\u_{\varepsilon}A[\Dv(\chi_{\O}\nabla\rho_{\varepsilon})]dxdt\\
&\quad +\varepsilon\int_{0}^{T}\int_{\O}\psi\eta\nabla\rho_{\varepsilon}\nabla\u_{\varepsilon}A[\rho_{\varepsilon}]dxdt
+\mu\int_{0}^{T}\int_{\O}\psi\u_{\varepsilon}\nabla\eta\rho_{\varepsilon}dxdt\\
&\quad -\mu\int_{0}^{T}\int_{\O}\psi\u_{\varepsilon}\nabla\eta\nabla A[\rho_{\varepsilon}]dxdt+\int_{0}^{T}\int_{\O}\psi\u_{\varepsilon}(\rho_{\varepsilon}R[\rho_{\varepsilon}\u_{\varepsilon}]
-\rho_{\varepsilon}\u_{\varepsilon}R[\rho_{\varepsilon}])dxdt\\
&\quad -\lambda\int_{0}^{T}\int_{\O}\left(\nabla\d_{\varepsilon}\odot\nabla\d_{\varepsilon}
-(\frac{1}{2}|\nabla\d_{\varepsilon}|^{2}+F(\d_{\varepsilon}))I_{3}\right)\psi\nabla\eta A[\rho_{\varepsilon}]dxdt\\
&\quad -\lambda\int_{0}^{T}\int_{\O}\left(\nabla\d_{\varepsilon}\odot\nabla\d_{\varepsilon}
-(\frac{1}{2}|\nabla\d_{\varepsilon}|^{2}+F(\d_{\varepsilon}))I_{3}\right)\psi\eta\nabla A[\rho_{\varepsilon}]dxdt
\end{split}\end{equation}
where $\chi_{\O}$ is the characteristics function of $\O$, $A[x]=\nabla\Delta^{-1}[x]$.

Meanwhile, we can use $$\varphi(t,x)=\psi(t)\eta(x)(\nabla\Delta^{-1})[\rho],\quad \psi\in \mathcal{D}(0,T),\quad \eta \in \mathcal{D}(\O),$$ as a test function for \eqref{V1b} to obtain
\begin{equation}\label{V4}\begin{split}
&\int_{0}^{T}\int_{\O}\varphi\eta(\bar{P}-\mu\Dv \u)\rho dxdt\\
&=\mu\int_{0}^{T}\int_{\O}\psi\nabla\u\nabla\eta A[\rho]dxdt-\int_{0}^{T}\int_{\O}\psi P\nabla\eta A[\rho]dxdt\\
&\quad -\int_{0}^{T}\int_{\O}\psi\rho\u\otimes\u\nabla\eta A[\rho]dxdt-\int_{0}^{T}\int_{\O}\psi_{t}\eta\rho\u A[\rho]dxdt\\
&\quad +\mu\int_{0}^{T}\int_{\O}\psi\u\nabla\eta\rho dxdt-\mu\int_{0}^{T}\int_{\O}\psi\u\nabla\eta\nabla A[\rho]dxdt\\
&\quad +\int_{0}^{T}\int_{\O}\psi\u(\rho R[\rho\u]-\rho\u R[\rho])dxdt\\
&\quad -\lambda\int_{0}^{T}\int_{\O}\left(\nabla\d\odot\nabla\d-(\frac{1}{2}|\nabla\d|^{2}+F(\d))I_{3}\right)\psi\nabla\eta A[\rho]dxdt\\
&\quad -\lambda\int_{0}^{T}\int_{\O}\left(\nabla\d\odot\nabla\d-(\frac{1}{2}|\nabla\d|^{2}+F(\d))I_{3}\right)\psi\eta\nabla A[\rho]dxdt.
\end{split}\end{equation}
For the related terms of $\rho_{\varepsilon},\u_{\varepsilon}$, following the same line in  \cite{FNP} we can show that these terms in \eqref{V3} converge  to their counterparts in \eqref{V4}. It remains to handle the terms related to $\d_{\varepsilon}$ in \eqref{V3}.
By virtue of the classical Mikhlin multiplier theorem (see  \cite{FNP}), we have
\begin{equation}
\label{v3+}
\nabla A[\rho_{\varepsilon}]\to \nabla A[\rho]\text{ in } C([0,T];L^{\beta}_{weak}(\O)) \quad \text{ as }\varepsilon \to 0,
\end{equation}
and
\begin{equation}
\label{v4+}
A[\rho_{\varepsilon}]\to A[\rho] \text { in } C(\overline{(0,T)\times\O))} \quad \text{ as }\varepsilon \to 0.
\end{equation}
Since
\begin{equation}
\begin{split}
\label{v5+}
& \int_{\O}\left|\nabla\d_{\varepsilon}\odot\nabla\d_{\varepsilon} A[\rho_{\varepsilon}]
-\nabla\d\odot\nabla\d A[\rho]\right|dx \\ &
\leq \int_{\O}|\nabla\d_{\varepsilon}|^{2}\left|A[\rho_{\varepsilon}]-A[\rho]\right|dx+\int_{\O}\left| \nabla\d_{\varepsilon}\right|\left|\nabla\d_{\varepsilon}
-\nabla\d\right||A[\rho]|dx\\
&\quad +\int_{\O}|\nabla\d|\left|\nabla\d_{\varepsilon}-\nabla\d\right||A[\rho]|dx,
\end{split}\end{equation}
%we can show the strong convergence of $\nabla\d_{\varepsilon} \text{ in } L^{2}(\O)$ similar to
%the strong convergence of $\nabla\d_{n} \text { in } L^{2}(\O).$
%Thus, from the strong convergence of $A[\rho_{\varepsilon}]$ together with H\"{o}lder's inequality, we have,
using H\"{o}lder's inequality to \eqref{v5+}, by \eqref{v3+}, \eqref{v4+}, and \eqref{v0c} we have
\begin{equation*}
\int_{0}^{T}\int_{\O}(\nabla\d_{\varepsilon}\odot\nabla\d_{\varepsilon})\psi\nabla\eta A[\rho_{\varepsilon}]dxdt\to \int_{0}^{T}\int_{\O}(\nabla\d\odot\nabla\d)\psi\nabla\eta A[\rho]dxdt\quad \text{ as }\varepsilon\to 0.
\end{equation*}
Similarly,
\begin{equation*}
\int_{0}^{T}\int_{\O}(\frac{1}{2}|\nabla\d_{\varepsilon}|^{2}I_{3})\psi\nabla\eta A[\rho_{\varepsilon}]dxdt\to
\int_{0}^{T}\int_{\O}(\frac{1}{2}|\nabla\d|^{2}I_{3})\psi\nabla\eta A[\rho]dxdt\quad \text{ as }\varepsilon\to 0.
\end{equation*}
Using the strong convergence of $F(\d_{\varepsilon})$, we conclude that,
\begin{equation*}\begin{split}
&\lambda\int_{0}^{T}\int_{\O}\left(\nabla\d_{\varepsilon}\odot\nabla\d_{\varepsilon}-
(\frac{1}{2}|\nabla\d_{\varepsilon}|^{2}+F(\d_{\varepsilon}))I_{3})\right)\psi\nabla\eta A[\rho_{\varepsilon}]dxdt \\
&\to \lambda\int_{0}^{T}\int_{\O}\left(\nabla\d\odot\nabla\d-(\frac{1}{2}|\nabla\d|^{2}
+F(\d))I_{3})\right)\psi\nabla\eta A[\rho]dxdt\quad \text{ as }\varepsilon\to 0.
\end{split}\end{equation*}
And similarly, \begin{equation*}\begin{split}
&\lambda\int_{0}^{T}\int_{\O}\left(\nabla\d_{\varepsilon}\odot\nabla\d_{\varepsilon}
-(\frac{1}{2}|\nabla\d_{\varepsilon}|^{2}+F(\d_{\varepsilon}))I_{3})\right)\psi\eta\nabla A[\rho_{\varepsilon}]dxdt  \\
& \to \lambda\int_{0}^{T}\int_{\O}\left(\nabla\d\odot\nabla\d-(\frac{1}{2}|\nabla\d|^{2}+F(\d))I_{3})\right)\psi\eta\nabla A[\rho]dxdt \quad \text{ as } \varepsilon\to 0.
\end{split}\end{equation*}
So we deduce that
\begin{equation*}\label{V5}\begin{split}
&\lim_{\varepsilon\to 0^{+}}\int_{0}^{T}\int_{\O}\psi\eta(\rho_{\varepsilon}^{\gamma}+
\delta\rho_{\varepsilon}^{\beta}-\mu\Dv\u_{\varepsilon})\rho_{\varepsilon}dxdt\\
&=\int_{0}^{T}\int_{\O}\psi\eta(\bar{P}-\mu\Dv\u)\rho dxdt\quad \text{ for any }\psi \in \mathcal{D}(0,T),\quad \eta \in \mathcal{D}(\O),
\end{split}\end{equation*}
where $\bar{P}=\overline{\rho^{\gamma}+\delta\rho^{\beta}}.$
The proof of Lemma \ref{Le6+} is complete.
\end{proof}

From Lemma \ref{Le6+}, we have
\begin{equation}\label{V6} \int_{0}^{T}\int_{\O}\rho\Dv\u-\overline{\rho\Dv\u}dxdt\leq\frac{1}{\mu}\int_{0}^{T}\int_{\O}(\bar{P}\rho -\overline{a\rho^{\gamma+1}+\delta\rho^{\beta+1}})dxdt.
\end{equation}
By\eqref{V2'} and \eqref{V6}, we deduce that
\begin{equation*}
\int_{\O}\overline{\rho\log(\rho)}-\rho\log(\rho)(\tau)dx\leq \frac{1}{\mu}\int_{0}^{T}\int_{\O}(\bar{P}\rho-\overline{a\rho^{\gamma+1}+\delta\rho^{\beta+1}})dxdt,
\end{equation*}
and $$\bar{P}\rho-\overline{\rho^{\gamma+1}+\delta\rho^{\beta+1}}\leq 0$$ due to the convexity of $\rho^{\gamma}+\delta\rho^{\beta}$.
So $$\int_{\O}\overline{\rho\log(\rho)}-\rho\log(\rho)(t)dx\leq 0.$$
On the other hand, $$\overline{\rho\log(\rho)}-\rho\log(\rho)\geq 0.$$
Consequently $\overline{\rho\log(\rho)}=\rho\log(\rho)$ that means
$$\rho_{\varepsilon}\to \rho \text{ in } L^{1}((0,T)\times\O).$$
Thus, we can pass to the limit as $\varepsilon \to 0$ to obtain the following result:

\begin{Proposition}\label{P2}
Assume $\O\subset \R^{3}$ is a bounded domain of class $C^{2+\vartheta},  \vartheta>0.$
 If there exists a constant $C_{0}>0,$ such that $ \d\cdot f(\d)\geq 0 \text{ for all } |\d|\geq C_{0}>0. $
 let $\delta>0$, and $$\beta>\max\left\{4,\frac{6\gamma}{2\gamma-3}\right\}$$
 be fixed. Then,  for any given $T>0$, there exists a finite energy weak solution $ (\rho,\u,\d)$ of the problem:
\begin{subequations}\label{V7}\begin{align}&\rho_{t}+\Dv(\rho\u)=0,\label{V7a}\\
&(\rho\u)_{t}+\Dv(\rho\u\otimes\u)+\nabla(\rho^{\gamma}+\delta\rho^{\beta}) \notag \\
&\qquad\qquad\qquad\qquad
 =\mu\Delta \u-\lambda\Dv\left(\nabla\d\odot\nabla\d-(\frac{1}{2}|\nabla\d|^{2}+F(\d))I_{3}\right),\label{V7b}\\
&\d_{t}+\u\cdot\nabla\d=\Delta\d-f(\d)\label{V7c}\end{align}\end{subequations}
with the boundary condition $\u|_{\partial\O} =0,\quad \d|_{\partial \O} =\d_{0}$ and initial condition \eqref{I2}. Moreover, $\rho\in L^{\beta+1}((0,T)\times\O)$ and the equation \eqref{V7a} holds in the sense of renormalized solutions on $D'((0,T)\times \R^{3})$ provided $\rho,\u$ were prolonged to be zero on $\R^{3}\setminus\O$.
Furthermore, $(\rho,\u,\d) $ satisfies the following uniform estimates:
\begin{equation}\label{V8}
\sup_{t\in[0,T]}\|\rho(t)\|_{L^{\gamma}(\O)}^{\gamma}\leq C E_{\delta}[\rho_{0},\m_{0},\d_{0}],
\end{equation}
\begin{equation}\label{V9}
\delta \sup_{t\in[0,T]}\|\rho(t)\|_{L^{\beta}(\O)}^{\beta}\leq C E_{\delta}[\rho_{0},\m_{0},\d_{0}],
\end{equation}
\begin{equation}\label{V10}
\sup_{t\in[0,T]}\|\sqrt{\rho(t)}\u(t)\|_{L^{2}(\O)}^{2}\leq C E_{\delta}[\rho_{0},\m_{0},\d_{0}],
\end{equation}
\begin{equation}\label{V11}
\|\u(t)\|_{L^{2}([0,T];H^{1}_{0}(\O))}\leq C E_{\delta}[\rho_{0},\m_{0},\d_{0}],
\end{equation}
\begin{equation}\label{V12}
\sup_{t\in[0,T]}\|\nabla\d\|_{L^{2}(\O)}^{2}\leq C E_{\delta}[\rho_{0},\m_{0},\d_{0}],
\end{equation}
\begin{equation}\label{V13}
\|\d\|_{L^{2}([0,T];H^{2}_{0}(\O))}\leq C E_{\delta}[\rho_{0},\m_{0},\d_{0}],
\end{equation}
where $C$ is independent of $\delta>0$ and
$$E_{\delta}[\rho_{0},\m_{0},\d_{0}]=\int_{\O}\left(\frac{1}{2}\frac{|\m_{0,\delta}|^{2}}{\rho_{0,\delta}}
+\frac{1}{\gamma-1}\rho_{0,\delta}^{\gamma}+\frac{\delta}{\beta-1}\rho_{0,\delta}^{\beta}+\frac{\lambda}{2}|\nabla\d_{0}|^{2}
+\lambda F(\d_{0})\right)dx.$$
\end{Proposition}

\begin{Remark}\label{R51}
Recalling the modified initial data \eqref{mid1}-\eqref{mid2}, we conclude that the modified energy $E_{\delta}[\rho_{0},\m_{0},\d_{0}]$ is bounded, and consequently the estimates in Proposition \ref{P2} hold independently of $\delta$.
\end{Remark}
\bigskip

%%%%%%%%%%%%%%%%%%%%
\section{Passing to the Limit in the Artificial Pressure Term}

The objective of this section is to recover the original system by vanishing the artificial pressure term.  Again in this part the crucial issue is to recover the strong convergence for $\rho_{\delta} $ in $L^{1}$ space.

\subsection{Better estimate of density}
Let us begin with a renormalized continuity equation
\begin{equation*}
b(\rho_{\delta})_{t}+\Dv(b(\rho_{\delta})\u_{\delta})+(b'(\rho_{\delta})\rho_{\delta}-b(\rho_{\delta}))\Dv\u_{\delta}=0 \quad \text { in } \mathcal{D'}((0,T)\times \R^{3})
\end{equation*}
for any uniformly bounded function $ b \in C^{1}[0,\infty).$
We can regularize the above equation as
\begin{equation}\label{Limit1}
\partial_{t}S_{m}[b(\rho)]+\Dv(S_{m}[b(\rho)]\u)+S_{m}[(b'(\rho)\rho-b(\rho))\Dv\u]=q_{m} \quad \text{ on } (0,T)\times \R^{3},
\end{equation}
where $S_{m}(v) $ denotes a spatial convolution with a family of regularizing kernels,
and
$$q_{m}\to 0 \text{ in } L^{2}(0,T;L^{2}(\R^{3})) \text { as } m\to \infty,$$
provided $b$ is uniformly bounded (see details in \cite{FNP}).
%%\begin{equation}\label{Limit1}\partial_{t}\rho^{\sigma}+\Dv(\rho^{\sigma}\u)+(\sigma\rho^{\sigma}-\rho^{\sigma})\Dv\u=0.\end{equation}

We use the operator $B$ to construct multipliers of the form
$$\varphi(t,x)=\psi(t)B[S_{m}[b(\rho_{\delta})]-\frac{1}{|\O|}\int_{\O}S_{m}[b(\rho_{\delta})]dx],
 \quad \psi \in \mathcal{D}(0,T),\quad 0 \leq \psi\leq 1,$$
where the operator $B$ was defined in Section 5.
Taking $b(\rho_{\delta})=\rho_{\delta}^{\sigma},$
 using \eqref{Limit1} and \eqref{V8}, with $\sigma $ small enough, we see that  $$S_{m}[\rho_{\delta}^{\sigma}]-\frac{1}{|\O|}\int_{\O}S_{m}[\rho_{\delta}^{\sigma}]dx$$
is in the space $C([0,T];L^{p}(\O))$ for any finite $p>1.$
%Recalling that the operator $B$ is a bounded linear operator, and thus
%$\varphi(t,x) \in C([0,T];L^{p}(\O))$ for any finite $p>1.$
 By \eqref{v1+} and the embedding theorem, we have $\varphi(t,x) \in C([0,T]\times\O)$. Consequently, $\varphi(t,x) $
can be used as a test function for \eqref{V7b}, then one arrives at the following formula:
\begin{equation*}\begin{split}
&\int_{0}^{T}\int_{\O}\psi(\rho_{\delta}^{\gamma}+\delta\rho_{\delta}^{\beta})S_{m}[\rho_{\delta}^{\sigma}]dxdt\\
&=\int_{0}^{T}\psi(t)\left(\int_{\O}(\rho_{\delta}^{\gamma}+
\delta\rho_{\delta}^{\beta})dx)(\frac{1}{|\O|}\int_{\O}S_{m}[\rho_{\delta}^{\sigma}]dx\right)dt
\\&\quad -\int_{0}^{T}\int_{\O}\psi_{t}\rho_{\delta}\u_{\delta}B[S_{m}[\rho_{\delta}^{\sigma}]-\frac{1}{|\O|}\int_{\O}S_{m}[\rho_{\delta}^{\sigma}]dx]dxdt\\
&\quad +\int_{0}^{T}\int_{\O}\psi(\mu\psi\nabla\u_{\delta}-\rho_{\delta}\u_{\delta}\otimes\u_{\delta})\nabla B[S_{m}[\rho_{\delta}^{\sigma}]-\frac{1}{|\O|}\int_{\O}S_{m}[\rho_{\delta}^{\sigma}]dx]dxdt\\
&\quad +\int_{0}^{T}\int_{\O}\psi\rho_{\delta}\u_{\delta}B[S_{m}(\rho_{\delta}^{\sigma}-\sigma\rho_{\delta}^{\sigma})\Dv\u
-\frac{1}{|\O|}\int_{\O}S_{m}[(\rho_{\delta}^{\sigma}-\sigma\rho_{\delta}^{\sigma})\Dv\u_{\delta}]dx]dxdt\\
&\quad -\int_{0}^{T}\int_{\O}\psi\rho_{\delta}\u_{\delta}B[\Dv S_{m}[(\rho_{\delta}^{\sigma}\u_{\delta})]]dxdt
 +\int_{0}^{T}\int_{\O}\psi\rho_{\delta}\u_{\delta}B[q_{m}-\frac{1}{|\O|}\int_{\O}q_{m}dx]dxdt\\
&\quad +\lambda\int_{0}^{T}\int_{\O}\psi\left(\nabla\d_{\delta}\odot\nabla\d_{\delta}-(\frac{1}{2}|\nabla\d_{\delta}|^{2}+F(\d_{\delta}))I_{3}\right)\nabla B[S_{m}[\rho_{\delta}^{\sigma}]-\frac{1}{|\O|}\int_{\O}S_{m}[\rho_{\delta}^{\sigma}]dx]dxdt \\
&=\sum_{i=1}^{6}I_{i}+\int_{0}^{T}\int_{\O}\psi\rho_{\delta}\u_{\delta}B[q_{m}-\frac{1}{|\O|}\int_{\O}q_{m}dx]dxdt.
\end{split}\end{equation*}
Noting that $q_{m}\to 0 \text{ in } L^{2}(0,T;L^{2}(\R^{3})) \text { as } m\to \infty$, we can pass to the limit for $m
\to \infty$ in the above equality to get the following:
$$\int_{0}^{T}\int_{\O}\psi(\rho_{\delta}^{\gamma+\sigma}+\delta\rho_{\delta}^{\beta+\sigma})dxdt\leq \sum_{i=1}^{6}|I_{i}|.$$

Now, we can estimate the integrals $I_{1}-I_{6}$ as follows.
\\(1) We see that \begin{equation*}
I_{1}=\int_{0}^{T}\psi(t)\left(\int_{\O}(a\rho_{\delta}^{\gamma}
+\delta\rho_{\delta}^{\beta})dx)(\frac{1}{|\O|}\int_{\O}S_{m}(\rho_{\delta}^{\sigma})dx\right)dt
\end{equation*}
is bounded uniformly in  $\delta$ provided $\sigma \leq \gamma$ by \eqref{V8} and \eqref{V9}.
\\ (2) As for the second term, by \eqref{V8}, \eqref{V10}, \eqref{V11} and together with the embedding $W^{1,p}(\O)\hookrightarrow L^{\infty}(\O)$ for $p>3,$ we have
\begin{equation*}\begin{split}
|I_{2}|&=\left|\int_{0}^{T}\int_{\O}\psi_{t}\rho_{\delta}\u_{\delta}B[S_{m}(\rho_{\delta}^{\sigma}) -\frac{1}{|\O|}\int_{\O}S_{m}(\rho_{\delta}^{\sigma})dx]dxdt\right|\\& \leq c \int_{0}^{T}|\psi_{t}|dt\leq C
\end{split}\end{equation*}
provided $\sigma \leq\frac{\gamma}{3}$.
\\(3) Similarly, for the third term, we have
\begin{equation*}
\begin{split}
|I_{3}|&=\left|\int_{0}^{T}\int_{\O}\psi(\mu\psi\nabla\u_{\delta}-\rho_{\delta}\u_{\delta}\otimes\u_{\delta})\nabla B[S_{m}(\rho_{\delta}^{\sigma})-\frac{1}{|\O|}\int_{\O}S_{m}(\rho_{\delta}^{\sigma})dx]dxdt\right|\\&\leq C
\end{split}\end{equation*}
if we choose $\sigma\leq \frac{\gamma}{2}$;
\\ (4) For $I_{4}$, by H\"{o}lder inequality, we have
\begin{equation*}\begin{split}
|I_{4}|&=\left|\int_{0}^{T}\int_{\O}\psi\rho_{\delta}\u_{\delta}B[S_{m}(\rho_{\delta}^{\sigma}-\sigma\rho_{\delta}^{\sigma})\Dv\u
-\frac{1}{|\O|}\int_{\O}S_{m}(\rho_{\delta}^{\sigma}-\sigma\rho_{\delta}^{\sigma})\Dv\u_{\delta}dx]dxdt\right|\\&
\leq C\int_{0}^{T}\|\rho_{\delta}\|_{L^{\gamma}(\O)}\|\u_{\delta}\|_{L^{6}(\O)}\|\rho_{\delta}^{\theta}\Dv\u_{\delta}\|_{L^{q}(\O)}dt,
\end{split}\end{equation*}
where $$p=\frac{6\gamma}{5\gamma-6},\quad\quad q=\frac{6\gamma}{7\gamma-6}.$$
 If we choose $\sigma\leq \frac{2\gamma}{3}-1$, and use \eqref{V8}, \eqref{V9} and \eqref{V11}, we conclude that $I_{4}$ is uniformly bounded.
\\ (5) Using the embedding inequality, we have
\begin{equation*}\begin{split}
|I_{5}|&=\left|\int_{0}^{T}\int_{\O}\psi\rho_{\delta}\u_{\delta}B[\Dv S_{m}(\rho_{\delta}^{\sigma}\u_{\delta})]dxdt\right|\\&\leq \int_{0}^{T}\|\rho_{\delta}\|_{L^{\gamma}}\|\u_{\delta}\|_{L^{6}}\|\rho_{\delta}^{\sigma}\u_{\delta}\|_{L^{p}}dt
\\ &\leq
C\int_{0}^{T}\|\rho_{\delta}\|_{L^{\gamma}}\|\u_{\delta}\|_{L^{6}}^{2}\|\rho_{\delta}^{\sigma}\|_{L^{r}}dt,
\end{split}\end{equation*}
where $r=\frac{3\gamma}{2\gamma-3}.$
If we choose $\sigma \leq \frac{2\gamma}{3}-1,$ and use \eqref{V8}, \eqref{V9} and \eqref{V11}, then $I_{5}$ is bounded.
\\(6) Finally, we estimate term $I_{6},$ let $\sigma\leq \frac{\gamma}{2},$ then
\begin{equation*}\begin{split}
&|I_{6}|\\&=\left|\lambda\int_{0}^{T}\int_{\O}\psi\left(\nabla\d_{\delta}\odot\nabla\d_{\delta}-
(\frac{1}{2}|\nabla\d_{\delta}|^{2}+F(\d_{\delta}))I_{3}\right)\nabla B[S_{m}(\rho_{\delta}^{\sigma})-\frac{1}{|\O|}\int_{\O}S_{m}(\rho_{\delta}^{\sigma})dx]dxdt\right|\\
&\leq C \int_{0}^{T}\|\nabla \d_{\delta}\|_{L^{4}(\O)}^{2}\|\nabla B[S_{m}(\rho_{\delta}^{\sigma})-\frac{1}{|\O|}\int_{\O}S_{m}(\rho_{\delta}^{\sigma})dx]\|_{L^{2}(\O)}dt\\
&\quad +C\int_{0}^{T}\|\nabla B[S_{m}(\rho_{\delta}^{\sigma})-\frac{1}{|\O|}\int_{\O}S_{m}(\rho_{\delta}^{\sigma})dx]\|_{L^{2}(\O)}dt
\\& \leq C,
\end{split}\end{equation*}
where we used the smoothness of $F$, \eqref{v1+}, \eqref{V8}, \eqref{V9} and $$\nabla\d_{\delta} \in L^{4}((0,T)\times\O).$$
%and$$\|B[f]\|_{W_{0}^{1,p}(\O)}\leq C \|f\|_{L^{p}(\O)} \text{ for any } 1<p<\infty,$$

All those above estimates together yield the following lemma:
\begin{Lemma}\label{Le5} Let $\gamma>\frac{3}{2}$.  There exists $\sigma>0$ depending only on $\gamma$, such that
$$\rho_{\delta}^{\gamma+\sigma}+\delta\rho_{\delta}^{\beta+\sigma} \text{ is bounded in } L^{1}((0,T)\times\O). $$
\end{Lemma}
%Using Lemma \ref{Le5} and Proposition \ref{P2}, we deduce that
%$$\delta\rho_{\delta}^{\beta}\to 0  \text{ in } L^{1}((0,T)\times\O) \text{ as } \delta\to 0,$$
%together with \eqref{App2a} and \eqref{App2b}, which imply that $ E_{\delta}[\rho_{0},m_{0},\d_{0}]$ in Proposition \eqref{P2} is independent of $\delta.$ Actually, we obtain that
%\begin{equation}E_{\delta}[\rho_{0},m_{0},\d_{0}]=E[\rho_{0},m_{0},\d_{0}].\end{equation}
\bigskip
\subsection{The limit passage}
By virtue of the estimates in Proposition \ref{P2} and Remark \ref{R51}, we can assume that, up to a subsequence if necessary,
\begin{equation}\label{Limit2}
\rho_{\delta}\to\rho \text{ in }C([0,T],L^{\gamma}_{weak}(\O)),
\end{equation}
\begin{equation}\label{Limit3}
\u_{\delta}\to\u \text{ weakly in } L^{2}([0,T];H^{1}_{0}(\O)),
\end{equation}
\begin{equation}\label{Limit4}
\d_{\delta}\to \d \text{ weakly in }L^{2}([0,T];H^{2}_{0}(\O))\cap L^{\infty}([0,T];H^{1}_{0}(\O)).
\end{equation}
\begin{equation}\label{L1+}
\d_{\delta}\to \d \text { strongly in } L^{2}(0,T;H^{1}(\O)),
\end{equation}
\begin{equation}\label{L2+}
\nabla\d_{\delta}\to\nabla\d \text { weakly in } L^{4}((0,T)\times\O),
\end{equation}
\begin{equation}\label{L3+}
\D\d_{\delta}-f(\d_{\delta})\to \D\d-f(\d) \text { weakly in } L^{2}(0,T;L^{2}(\O)),
\end{equation}
\begin{equation}\label{L4+}
F(\d_{\delta})\to F(\d) \text { strongly in } L^{2}(0,T;H^{1}(\O)).
\end{equation}
Letting $\delta\to 0,$  we have,
\begin{equation}\label{L1++}
\rho_{\delta}^{\gamma}\to \overline{\rho^{\gamma}} \text{ weakly in } L^{1}((0,T)\times(\O)),
\end{equation}
subject to a subsequence.

From \eqref{L1+} and \eqref{L4+}, we have, as $\delta \to 0,$
\begin{equation}
\begin{split}\label{L2++}
&\nabla\d_{\delta}\odot\nabla\d_{\delta}-(\frac{1}{2}|\nabla\d_{\delta}|^{2}+F(\d_{\delta}))I_{3}\\
&\to \nabla\d\odot\nabla\d-(\frac{1}{2}|\nabla\d|^{2}+F(\d))I_{3}\quad \text{ in }\mathcal{D^{'}}(\O\times(0,T)),
\end{split}
\end{equation}
and
\begin{equation}\label{L3++}
\u_{\delta}\cdot\nabla\d_{\delta}\to \u\cdot\nabla\d \quad \text{ in }\mathcal{D^{'}}(\O\times(0,T)),
\end{equation}
as $\delta\to 0.$

On the other hand, by virtue of \eqref{V7b}, \eqref{V8}-\eqref{V11}, we obtain
\begin{equation}\label{L4++}
\rho_{\delta}\u_{\delta}\to \rho\u \text { in } C([0,T];L^{\frac{2\gamma}{\gamma+1}}_{weak}(\O)).
\end{equation}
Similarly, we have, as $\delta\to 0$,
\begin{equation*}
\d_{\delta}\to\d \text{ in } C([0,T];L^{2}_{weak}(\O)).
\end{equation*}
By Lemma \ref{Le5}, we get $$\delta\rho_{\delta}^{\beta}\to 0  \text{ in } L^{1}((0,T)\times\O) \text{ as } \delta\to 0.$$
Thus, the limit of $(\rho, \rho\u, \d )$ satisfies the initial and boundary conditions of \eqref{I2} and \eqref{I3}.

Since $\gamma>\frac{3}{2},$ \eqref{Limit3} and \eqref{L4++} combined with the compactness of $H^{1}(\O)\hookrightarrow L^{2}(\O)$ imply, as $\delta\to 0$,
$$\rho_{\delta}\u_{\delta}\otimes\u_{\delta}\to \rho\u\otimes\u \quad\quad\quad\text{ in } \mathcal{D'}((0,T)\times\O).$$
Consequently, letting $\delta\to 0$ in \eqref{V7} and making use of \eqref{Limit2}-\eqref{L4++},
the limit of $(\rho_{\delta},\u_{\delta},\d_{\delta})$ satisfies the following  system:
\begin{subequations}
\label{Limit6}
\begin{align}
&\rho_{t}+\Dv(\rho\u)=0,\label{Limit6a}\\
&(\rho\u)_{t}+\Dv(\rho\u\otimes\u)+\nabla\overline{\rho^{\gamma}}
=\mu\Delta\u-\lambda\Dv\left(\nabla\d\odot\nabla\d-(\frac{1}{2}|\nabla\d|^{2}+F(\d))I_{3}\right) \label{Limit6b}\\
&\d_{t}+\u\cdot\nabla\d=\Delta\d-f(\d)
\label{Limit6c}\end{align}\end{subequations}
in $\mathcal{D^{'}}(\O\times(0,T)).$
\subsection{The strong convergence of density}
In order to complete the proof of Theorem \ref{T1}, we still need to show the strong convergence of $\rho_{\delta}$ in $L^{1}(\O),$ or, equivalently $\bar{\rho^{\gamma}}=\rho^{\gamma}.$

 Since $\rho_{\delta},\u_{\delta}$ is a renormalized solution of the equation \eqref{Limit6a} in $\mathcal{D^{'}}((0,T)\times \R^{3})$, we have
\begin{equation*}
T_{k}(\rho_{\delta})_{t}+\Dv(T_{k}(\rho_{\delta}\u_{\delta}))+(T_{k}(\rho_{\delta})\rho_{\delta}
-T_{k}(\rho_{\delta}))\Dv(\u_{\delta})=0 \text{ in } \mathcal{D^{'}}((0,T)\times R^{3}),
\end{equation*}
where $T_{k}(z)=kT(\frac{z}{k}) \text{ for } z\in \R,\quad k=1,2,3\cdots $
and $T\in C^{\infty}(\R)$ is chosen so that
$$T(z)=z \text{ for } z\leq1,\quad T(z)=2 \text { for } z\geq 3,\quad T \text{ convex}.$$
Passing to the limit for $\delta\to 0$ we deduce that
\begin{equation*}
\partial_{t}\overline{T_{k}(\rho)}+\Dv(\overline{(T_{k}(\rho))}\u)+\overline{(T_{k}^{'}(\rho)\rho-T_{k}(\rho))\Dv\u}=0 \quad\text{ in } \mathcal{D^{'}}((0,T)\times \R^{3})),
\end{equation*}
where
$$T_{k}^{'}(\rho{\delta})\rho_{\delta}-T_{k}(\rho_{\delta})\Dv\u_{\delta}\to \overline{(T_{k}^{'}(\rho)\rho-T_{k}(\rho))\Dv\u}\quad \text { weakly in } L^{2}((0,T)\times\O),$$
and$$T_{k}(\rho_{\delta})\to \overline{T_{k}(\rho)} \text{ in } C([0,T];L^{p}_{weak}(\O)) \text { for all }1\leq p<\infty.$$
Using the function
$$\varphi(t,x)=\psi(t)\eta(x)A_{i}[T_{k}(\rho_{\delta})],\quad \psi\in \mathcal{D}[0,T],\quad \eta\in \mathcal{D}(\O),$$
as a test function for \eqref{V7b},
by a similar calculation to the previous sections, we can deduce the following result:
\begin{Lemma}\label{Le6}
Let $(\rho_{\delta},\u_{\delta})$ be the sequence of approximate solutions constructed in Proposition \ref{P2}, then
\begin{equation*}
\lim_{\delta\to 0}\int_{0}^{T}\int_{\O}\psi\eta(\rho_{\delta}^{\gamma}-
\mu\Dv\u_{\delta})T_{k}(\rho_{\delta})dxdt=\int_{0}^{T}\int_{\O}\psi\eta(\overline{\rho^{\gamma}}
-\mu\Dv\u)\overline{T_{k}(\rho)}dxdt
\end{equation*}
for any $\psi \in \mathcal{D}(0,T), \eta \in \mathcal{D}(\O).$
\end{Lemma}
In order to get the strong convergence of $\rho_{\delta},$ we need to define the oscillation defect measure as follows:
\begin{equation*}
OSC_{\gamma+1}[\rho_{\delta}\to\rho]((0,T)\times\O)=\sup_{k\geq 1}\lim_{\delta\to 0}\sup\int_{0}^{T}\int_{\O}|T_{k}(\rho_{\delta})-T_{k}(\rho))|^{\gamma+1}dxdt.
\end{equation*}
Here we state a lemma about the oscillation defect measure:

\begin{Lemma}\label{Le7}
There exists a constant $C$ independent of $k$ such that
\begin{equation*}
OSC_{\gamma+1}[\rho_{\delta}\to\rho]((0,T)\times\O) \leq C
 \end{equation*}
for any $k \geq 1.$
\end{Lemma}

\begin{proof}
Following the line of argument presented in  \cite{FNP}, and by  Lemma \ref{Le6}, we obtain
\begin{equation*}
OSC_{\gamma+1}[\rho_{\delta}\to\rho]((0,T)\times\O)
\leq \lim_{\delta\to 0}\int_{0}^{T}\int_{\O}\Dv\u_{\delta}T_{k}(\rho_{\delta})-\Dv\u\overline{T_{k}(\rho)}dxdt.
\end{equation*}
On the other hand,
\begin{equation*}\begin{split}
&\lim_{\delta\to 0}\int_{0}^{T}\int_{\O}\Dv\u_{\delta}T_{k}(\rho_{\delta})-\Dv\u\overline{T_{k}(\rho)}dxdt\\
&=\lim_{\delta\to 0}\int_{0}^{T}\int_{\O}(T_{k}(\rho_{\delta})-T_{k}(\rho)+T_{k}(\rho)-\overline{T_{k}(\rho)})\Dv\u_{\delta}dxdt\\
&\leq 2\sup_{\delta}\|\nabla\u_{\delta}\|_{L^{2}((0,T)\times\O)}\lim_{\delta\to 0}\sup\|T_{k}(\rho_{\delta})-T_{k}(\rho)\|_{L^{2}((0,T)\times\O)}.
\end{split}\end{equation*}
So we can conclude the Lemma.
\end{proof}

We are now ready to show the strong convergence of the density. To this end, we introduce a sequence of functions $L_{k} \in C^{1}(\R):$
\begin{equation*}
 L_{k}(z)=\begin{cases}z ln z,\quad 0\leq z<k\\zln(k)+z\int_{k}^{z}\frac{T_{k}(s)}{s^{2}}ds, \quad z\geq k.\end{cases}
\end{equation*}
Noting that $L_{k}$ can be written as $$L_{k}(z)=\beta_{k}z+b_{k}z,$$
where $b_{k}$ satisfy \eqref{EI9}, we deduce that
\begin{equation}\label{Limit7}
\partial_{t}L_{k}(\rho_{\delta})+\Dv(L_{k}(\rho_{\delta})\u_{\delta})+T_{k}(\rho_{\delta})\Dv\u_{\delta}=0,
\end{equation}
and \begin{equation}\label{Limit8}
\partial_{t}L_{k}(\rho)+\Dv(L_{k}(\rho)\u)+T_{k}(\rho)\Dv\u=0
\end{equation}
in $\mathcal{D^{'}}((0,T)\times\O).$
Letting $\delta\to 0,$  we can assume that
\begin{equation*}
L_{k}(\rho_{\delta})\to\overline{L_{k}(\rho)} \text { in } C([0,T];L^{\gamma}_{weak}(\O)).
\end{equation*}
Taking the difference of  \eqref{Limit7} and \eqref{Limit8},  and integrating with respect to time $t$, we obtain
\begin{equation*}\begin{split}
&\int_{\O}(L_{k}(\rho_{\delta})-L_{k}(\rho))\phi dx\\
&=\int_{0}^{T}\int_{\O}\Big((L_{k}(\rho_{\delta})\u_{\delta}-L_{k}(\rho)\u)\cdot\nabla \phi+(T_{k}(\rho)\Dv\u-T_{k}(\rho_{\delta})\Dv\u_{\delta})\phi\Big)dxdt,
\end{split}\end{equation*}
for any $\phi \in \mathcal{D}(\O).$
Following the line of argument in \cite{FNP}, we get
\begin{equation}\label{Limit9}\begin{split}
&\int_{\O}\left(\overline{L_{k}(\rho)}-L_{k}(\rho)\right)(t)dx\\
&=\int_{0}^{T}\int_{\O}T_{k}(\rho)\Dv\u dx dt
-\lim_{\delta\to 0^{+}}\int_{0}^{T}\int_{\O}T_{k}(\rho_{\delta})\Dv\u_{\delta}dxdt.
\end{split}\end{equation}
We observe that the term $\overline{L_{k}(\rho)}-L_{k}(\rho)$ is bounded by its definition.
Using Lemma \ref{Le7} and the monotonicity of the pressure, we can estimate the right-hand side of \eqref{Limit9}:
\begin{equation}\label{Limit10}\begin{split}
&\int_{0}^{T}\int_{\O}T_{k}(\rho)\Dv\u dx dt-\lim_{\delta\to 0^{+}}\int_{0}^{T}\int_{\O}T_{k}(\rho_{\delta})\Dv\u_{\delta} dxdt\\
&\leq\int_{0}^{T}\int_{\O}(T_{k}(\rho)-\overline{T_{k}(\rho)})\Dv\u dxdt.
\end{split}\end{equation}
By virtue of Lemma \ref{Le7}, the right-hand side of \eqref{Limit10} tends to zero as $k\to \infty.$
So we conclude that $$\overline{\rho\log(\rho)(t)}=\rho\log(\rho)(t)$$ as $k\to \infty.$
Thus we obtain  the strong convergence of $\rho_{\delta}$ in $L^{1}((0,T)\times\O).$

Therefore we complete the proof of Theorem \ref{T1}.

\bigskip

%%%%%%%%%%%%%%%%%
\section{Large-Time Behavior of Weak Solutions}

The aim of this section is to study the large-time behavior of the finite energy weak solutions obtained in Theorem \ref{T1}.

%The system \eqref{I1} represents a gradient flow with the total energy
%$$E(t)=\int_{\O}\left(\frac{1}{2}\rho|\u|^{2}+\frac{\rho^{\gamma}}{\gamma-1}+\lambda\frac{1}{2}|\nabla\d|^{2}+\lambda F(\d)\right)dx$$
%as a Lyapunov functional satisfying the energy inequality:
First of all, from Theorem \ref{T1}, we have
\begin{equation}\label{lar1}
\ess_{t>0}E(t)+\int_{0}^{T}\int_{\O}\left(\mu|\nabla\u|^{2}+\lambda|\D\d-f(\d)|^{2}\right)dxdt\leq E(0),
\end{equation}
where $$E(t)=\int_{\O}\left(\frac{1}{2}\rho|\u|^{2}+\frac{\rho^{\gamma}}{\gamma-1}+\lambda\frac{1}{2}|\nabla\d|^{2}+\lambda F(\d)\right)dx.$$
Following the  argument in \cite{ FP}, we take a sequence
$$\rho_{m}(t,x):=\rho(t+m,x);$$
$$\u_{m}(t,x):=\u(t+m,x);$$
$$\d{m}(t,x):=\d(t+m,x),$$
for all integer $m$, and $t\in (0,1), x\in \O$.

From \eqref{lar1}, we have
$$\rho_{m}\in L^{\infty}([0,1];L^{\gamma}(\O)),\quad \d_{m}\in L^{\infty}([0,1];H^{1}(\O)),$$
$$\sqrt{\rho_{m}}\u_{m}\in L^{\infty}([0,1];L^{2}(\O)),\quad \rho_{m}\u_{m}\in L^{\infty}([0,1];L^{\frac{2\gamma}{\gamma+1}}(\O)),$$
which are independent of $m$.
Moreover, we have \begin{equation}\label{lar2}
\lim_{m\to\infty}\int_{0}^{1}\left(\|\nabla\u_{m}\|_{L^{2}(\O)}+\|\D \d_{m}-f(\d_{m})\|_{L^{2}(\O)}\right)dt=0.
\end{equation}
So we can assume that, up to a subsequence if necessary, as $m\to \infty$,
$$\rho_{m}(t,x)\to \rho_{s} \text{ weakly in } L^{\gamma}((0,1)\times\O);$$
$$\u_{m}(t,x) \to \u_{s} \text { weakly in } L^{2}([0,1];H^{1}_{0}(\O));$$
$$\D\d_{m}-f(\d_{m})\to \D\d_{s}-f(\d_{s}) \text { weakly in } L^{2}([0,1];L^{2}(\O)).$$
Furthermore, $$\int_{\O}\rho_{s}dx\leq \lim_{m\to\infty}\inf\int_{\O}\rho_{m}dx\leq C(E_{0}).$$
Using \eqref{lar2} and the Poincar\'e inequality, we have $$\lim_{m\to \infty}\int_{0}^{1}\|\u_{m}\|_{L^{2}(\O)}^{2}dt=0.$$
By the embedding of $H^{1}\hookrightarrow L^{2}$, we have
\begin{equation*}
\u_{s}=0 \text { almost everythere in } (0,1)\times\O.
\end{equation*}
From \eqref{lar2} again, we have
\begin{equation}\label{lar3}
\D\d_{s}-f(\d_{s})=0 \text { almost everythere in } (0,1)\times\O,
\end{equation}
under the  boundary condition
\begin{equation}\label{lar1+}
\d_{s}|_{\partial\O}=\d_{0}.
\end{equation}
By the elliptic theory, there exist a  unique solution $\d_{s}\in C^{2}(\O)\cap C(\overline{\O})$ to \eqref{lar3} and \eqref{lar1+}.

On one hand, from \eqref{lar1} and \eqref{lar2}, we have
\begin{equation}\label{lar4}
\lim_{m\to\infty}\int_{0}^{1}\left(\|\rho_{m}|\u_{m}|^{2}\|_{L^{\frac{3\gamma}{\gamma+3}}(\O)}
+\|\rho_{m}|\u_{m}|\|_{L^{\frac{6\gamma}{\gamma+6}}(\O)}^{2}\right)dt=0.
\end{equation}
Since $\rho,\u$ are the solutions to \eqref{I1a} in the sense of renormalized solutions, one has, in particular,
\begin{equation}
\label{lar2+}
\rho_{t}+\Dv(\rho\u)=0 \quad \text{ in } \mathcal{D'}((0,T)\times\O).
\end{equation}
Taking a test function $\varphi(t,x)=\psi(t)\phi(x)$ in \eqref{lar2+}, where $\psi(t)\in \mathcal{D}(0,1),\; \phi\in \mathcal{D}(\O),$
we have $$\int_{0}^{1}\left(\int_{\O}\rho_{m}\phi dx\right)\psi'(t)dt+\int_{0}^{1}\int_{\O}\rho_{m}\u_{m}\nabla\phi\psi dxdt=0.$$
Letting $m\to \infty$ and using \eqref{lar4}, we have
$$\int_{0}^{1}\left(\int_{\O}\rho_{s}\phi dx\right)\psi'(t)dt=0,$$
which means that $\rho_{s}$ is independent of time $t$.

Similar to Lemma \ref{Le5}, we have
$$\rho_{m}^{\gamma+\a} \text { is bounded in } L^{1}((0,1)\times\O) \quad \text { independently of } m>0,$$for some $\a>0.$
So we conclude that
\begin{equation}\label{lar5}
\rho_{m}^{\gamma}\to \overline{\rho^{\gamma}} \text { weakly in } L^{1}((0,1)\times\O).
\end{equation}
Therefore, passing to the limit in \eqref{I1b} and using \eqref{lar2}, \eqref{lar4}, we obtain
\begin{equation}\label{lar6}
\nabla\overline{\rho^{\gamma}}=-\lambda\Dv\left(\nabla\d_{s}\odot\nabla\d_{s}-
(\frac{1}{2}|\nabla\d_{s}|^{2}+F(\d_{s}))I_{3}\right) \text { in } \mathcal{D'}(\O).
\end{equation}
where $\d_{s}$ is the solution to \eqref{lar3} with its boundary condition \eqref{lar1+}.

On the other hand, we can use $L^{p}-$version of celebrated div-curl lemma to show that the convergence in \eqref{lar5} is strong. We refer the readers to \cite{FP} and \cite{HW} for details.
Due to the strong convergence in \eqref{lar5}, we have $$\rho_{m}\to\rho_{s} \text { strongly in } L^{\gamma}((0,1)\times\O).$$
This, combined with \eqref{lar5} and \eqref{lar6}, gives $$\nabla\rho_{s}^{\gamma}=-\lambda\Dv\left(\nabla\d_{s}\odot\nabla\d_{s}-(\frac{1}{2}|\nabla\d_{s}|^{2}+F(\d_{s}))I_{3}\right) $$ in the sense of distributions.
Denoting $$H=\rho_{s}^{\gamma}-\lambda F(\d_{s}),$$
and using \eqref{lar3} to rewrite the above equation as follows:
$$\nabla H=-\lambda\nabla\d_{s}f(\d_{s}).$$
Notice that $$f(\d_{s})=\nabla_{\d_{s}}F(\d_{s}),$$ we have $$\nabla(H+\lambda F(\d_{s}))=0,$$
that is,$$\nabla\rho_{s}^{\gamma}=0.$$
%thus $$H+\lambda F(\d_{s})=c.$$
%This yields to $$\rho_{s}=c^{\frac{1}{\gamma}},$$
%where $c>0$ is a constant and uniquely up to $m(\rho_{s})=\int_{\O}\rho_{s}dx.$

Finally, by the energy inequality, the energy converges to a finite constant as $t\to \infty$:
$$E_{\infty}:=\overline{\lim_{t\to \infty}}E(t),$$
by \eqref{lar4}, we have $$\lim_{m\to\infty}\int_{m}^{m+1}\int_{\O}\rho|\u|^{2}dxdt=0.$$
Thus,\begin{equation*}\begin{split}
&E_{\infty}=\overline{\lim_{m\to\infty}}\int_{m}^{m+1}\int_{\O}\left(\frac{1}{2}\rho|\u|^{2}+
\frac{1}{\gamma-1}\rho^{\gamma}+\lambda\frac{1}{2}|\nabla\d|^{2}+\lambda F(\d)\right)dxdt
\\&=\int_{\O}\left(\frac{1}{\gamma-1}\rho_{s}^{\gamma}+\lambda \frac{1}{2}|\nabla\d_{s}|^{2}+\lambda F(\d_{s})\right)dx.
\end{split}\end{equation*}
We observe that
$$\lim_{m\to\infty}\int_{m}^{m+1}\int_{\O}F(d)dxdt=\int_{\O}F(\d_{s})dx$$
and $$\lim_{m\to\infty}\int_{m}^{m+1}\int_{\O}\frac{1}{2}|\nabla\d|^{2}dxdt=\int_{\O}\frac{1}{2}|\nabla\d_{s}|^{2}dx$$
because of \eqref{lar3}.
Moreover, using \eqref{I1a} one can easily see that
$$\rho(t,x)\to\rho_{s} \text { weakly in } L^{\gamma}(\O) \text { as } t\to \infty.$$
Thus, we have
\begin{equation*}\begin{split}
&E_{\infty}:=\int_{\O}\left(\frac{1}{\gamma-1}\rho_{s}^{\gamma}+\lambda\frac{1}{2}|\nabla\d_{s}|^{2}+\lambda F(\d_{s})\right)dx\\& \leq \lim_{m\to\infty}\inf\int_{\O}\left(\frac{1}{\gamma-1}\rho^{\gamma}+\lambda \frac{1}{2}|\nabla\d|^{2}+\lambda F(\d)\right)dx \\& \leq \lim_{t\to\infty}\sup\int_{\O}\left(\frac{1}{\gamma-1}\rho_{s}^{\gamma}+\lambda \frac{1}{2}|\nabla\d_{s}|^{2}+\lambda
F(\d_{s})\right)dx\\& \leq\overline{\lim_{t\to\infty}}\int_{\O}\left(\frac{1}{2}\rho|\u|^{2}
+\frac{1}{\gamma-1}\rho_{s}^{\gamma}+\lambda\frac{1}{2}|\nabla\d_{s}|^{2}+\lambda F(\d_{s})\right)dx\\&
 =\overline{\lim_{t\to\infty}}E(t)=E_{\infty},
\end{split}\end{equation*}
which means $$\lim_{t\to\infty}\int_{\O}\frac{1}{\gamma-1}\rho^{\gamma}dx=\int_{\O}\frac{1}{\gamma-1}\rho_{s}^{\gamma}dx,$$
\eqref{Tlar1} follows since the space $L^{\gamma}$ is uniformly convex.

The proof of Theorem \ref{T2} is complete.

\bigskip

\section*{Acknowledgments}

D. Wang's research was supported in part by the National Science
Foundation under Grant DMS-0906160 and by the Office of Naval
Research under Grant N00014-07-1-0668. C. Yu's research was supported in part by the National Science
Foundation under Grant DMS-0906160.

\end{document}